\documentclass[journal]{IEEEtran}
\usepackage{amsmath}
\usepackage{graphicx}
\usepackage{booktabs}
\usepackage{alltt}
\usepackage{boxedminipage}
\usepackage{epstopdf}
\usepackage{subfigure}
\usepackage{amsmath}
\usepackage{algorithm}
\usepackage{algpseudocode}
\usepackage{cite}
\usepackage{multicol}
\usepackage{multirow}
\usepackage{mathtools}
\usepackage{amssymb}
\usepackage{amsthm}
\usepackage{color}
\usepackage{bm}
\usepackage{indentfirst}
\usepackage{url}
\usepackage{nomencl}
\usepackage{booktabs}
\usepackage{cite}
\usepackage[numbers,sort&compress]{natbib}
\usepackage{amsmath}
\usepackage{graphicx}
\usepackage{subfigure}
\usepackage{amsmath}
\usepackage{algorithm}
\usepackage{algpseudocode}
\usepackage{cite}
\definecolor{mygray}{gray}{.9}
\usepackage{multicol}
\usepackage{multirow}
\usepackage{mathtools}
\usepackage{amssymb}
\usepackage{color}
\usepackage{bm}
\usepackage{indentfirst}
\usepackage{url}
\usepackage{nomencl}
\usepackage{booktabs}
\usepackage{cite}
\usepackage[numbers,sort&compress]{natbib}

\newtheorem{Theorem}{Theorem}
\newtheorem{Lemma}{Lemma}
\newtheorem{Proposition}{Proposition}

\newtheorem{Definition}{Definition}
\newtheorem{Remark}{Remark}
\newtheorem{Assumption}{Assumption}
\allowdisplaybreaks

\begin{document}

\title{Composite Optimization with Coupling Constraints via Penalized Proximal Gradient Method in Partially Asynchronous Networks}

\author{Jianzheng Wang and Guoqiang Hu
\thanks{Jianzheng Wang and Guoqiang Hu are with the School of Electrical and Electronic Engineering, Nanyang Technological University, Singapore 639798 (email:
        {\tt\small wang1151@e.ntu.edu.sg, gqhu@ntu.edu.sg}).}%
}
%

\maketitle

\begin{abstract}
In this paper, we consider a composite optimization problem with linear coupling constraints in a multi-agent network. In this problem, the agents cooperatively optimize a global composite cost function which is the linear sum of individual cost functions composed of smooth and possibly non-smooth components. To solve this problem, we propose an asynchronous penalized proximal gradient (Asyn-PPG) algorithm, a variant of classical proximal gradient method, with the presence of the asynchronous updates of the agents and communication delays in the network. Specifically, we consider a slot-based asynchronous network (SAN), where the whole time domain is split into sequential time slots and each agent is permitted to execute multiple updates during a slot by accessing the historical state information of others. Moreover, we consider a set of global linear constraints and impose violation penalties on the updating algorithms. By the Asyn-PPG algorithm, it shows that a periodic convergence with rate $\mathcal{O}(\frac{1}{K})$ ($K$ is the index of time slots) can be guaranteed with the coefficient of the penalties synchronized at the end of each time slot. The feasibility of the proposed algorithm is verified by solving a consensus based distributed LASSO problem and a social welfare optimization problem in the electricity market respectively.
\end{abstract}


\IEEEpeerreviewmaketitle
\section{Introduction}

\subsection{Background and Motivation}\label{bnm}
\IEEEPARstart{I}{n} recent years, decentralized optimization problems have been extensively investigated in different research fields, such as distributed control of multi-robot systems, decentralized regularization problems with massive data sets, and economic dispatch problems in power systems \cite{SC20,margellos2017distributed,bai2017distributed}.
In those problems, there are two main categories of how information and agent actions are managed: synchronous and asynchronous. In a synchronous system, certain global clock for agent interactions and activations is established to ensure the correctness of optimization result \cite{arjomandi1983efficiency}. However, in many decentralized systems, there is no such a guarantee. The reasons mainly lie in the following two aspects.
\begin{enumerate}
  \item {\em{Asynchronous Activations:}} In some multi-agent systems, each agent may only be keen on his own updates regardless of the process of others. Such an action pattern may cause an asynchronous computation environment. For example, some agents with higher computation capacity may take more actions during a given time horizon without ``waiting for'' the slow ones \cite{tseng1991rate}.
  \item {\em{Communication Delays:}} In synchronous networks, the agents are assumed to access the up-to-date information without any packet loss. This settlement requires an efficient communication process or reserving a ``zone'' between two successive updates for the data transmission. However, in large-scale decentralized systems, complete synchronization of communications may be costly if the delay is large and computational frequency is high \cite{hannah2018unbounded}.
\end{enumerate}

In addition, we consider a composite optimization problem with coupling constraints, where the objective function is separable and composed of smooth and possibly non-smooth components. The concerned problem structure arises from various fields, such as logistic regression, boosting, and support vector machines \cite{kleinbaum2002logistic,schapire2013boosting,hearst1998support}. Observing that proximal gradient method takes the advantage of some simple-structured composite functions and is usually numerically more stable than the subgradient counterpart \cite{bertsekas2011incremental}, in this paper, we aim to develop a decentralized proximal gradient based algorithm for solving the composite optimization problem in an asynchronous\footnote[1]{In this paper, ``asynchronous'' may be referred to as ``partially asynchronous'' which could be different from ``fully asynchronous'' studied in some works (e.g., \cite{baudet1978asynchronous}) without more clarification for briefness. The former definition may contain some mild assumptions on the asynchrony, e.g., bounded communication delays.} network.


\subsection{Literature Review}\label{881}

Proximal gradient method is related to the proximal minimization algorithm which was studied in early works \cite{martinet1970regularisation,rockafellar1976monotone}. By this method, a broad class of composite optimization problems with simple-structured objective functions can be solved efficiently \cite{bao2012real,chen2012smoothing,banert2019general}.
\cite{aybat2017distributed,hong2017stochastic,li2019decentralized} further studied the decentralized realization of proximal gradient based algorithms.
Decentralized proximal gradient methods dealing with global linear constraints were discussed in \cite{li2017convergence,wang2021distributed,ma2016alternating,jiang2012inexact}. \cite{beck2009fast,chen2012fast,li2015accelerated} present some accelerated versions of proximal gradient method. Different from the existing works, we will show that by our proposed penalty based Asyn-PPG algorithm, a class of composite optimization problems with coupling constraints can be solved asynchronously in the proposed SAN, which enriches the exiting proximal gradient methods and applications.

To deal with the asynchrony of multi-agent networks, existing works usually capture two factors: asynchronous action clocks and unreliable communications \cite{tsitsiklis1986distributed}. In those problems, the decentralized algorithms are built upon stochastic or deterministic settings depending on whether the probability distribution of the asynchronous factors is utilized. Among those works, stochastic optimization based models and algorithms are fruitful \cite{xu2017convergence,notarnicola2016asynchronous,bastianello2020asynchronous,mansoori2019fast,wei20131}. For instance,
in \cite{xu2017convergence}, an asynchronous distributed gradient method was proposed for solving a consensus optimization problem by considering random communications and updates. The authors of \cite{notarnicola2016asynchronous} proposed a randomized dual proximal gradient method, where the agents execute node-based or edge-based asynchronous updates activated by local timers. An asynchronous relaxed ADMM algorithm was proposed in \cite{bastianello2020asynchronous} for solving a distributed optimization problem with asynchronous actions and random communication failures.

All the optimization algorithms in \cite{xu2017convergence,wei20131,notarnicola2016asynchronous,bastianello2020asynchronous,mansoori2019fast} require the probability distribution of asynchronous factors to establish the parameters of algorithms and characterize convergence properties. However, in practical applications, the probability distributions may be difficult to acquire and would cause inaccuracy issues in the result due to the limited historical data \cite{bertsekas2002introduction}. To overcome those drawbacks, some works leveraging on deterministic analysis arose in the recent few decades \cite{bertsekas2011incremental,chazan1969chaotic,bertsekas1989parallel,zhou2018distributed,kibardin1980decomposition,nedich2001distributed,hong2017distributed,kumar2016asynchronous,tian2020achieving,hale2017asynchronous,cannelli2020asynchronous}. For instance, in \cite{chazan1969chaotic}, a chaotic relaxation method was studied for solving a quadratic minimization problem by allowing for both asynchronous actions and communication delays, which can be viewed as a prototype of a class of asynchronous problems. The authors of \cite{bertsekas1989parallel} further investigated the asynchronous updates and communication delays in a routing problem in data networks based on deterministic relaxations. The authors of \cite{zhou2018distributed} proposed an m-PAPG algorithm in asynchronous networks by employing proximal gradient method in machine learning problems with a periodically linear convergence guarantee.

Another line of asynchronous optimizations with deterministic analysis focuses on incremental (sub)gradient algorithms, which can be traced back to \cite{kibardin1980decomposition}. In more recent works, a wider range of asynchronous factors have been explored. For example, in \cite{nedich2001distributed}, a cluster of processors compute the subgradient of their local objective functions triggered by asynchronous action clocks. Then, a master processor acquires all the available but possibly outdated subgradients and updates its state for the subsequent round. The author of \cite{bertsekas2011incremental} proposed an incremental proximal method, which admits a fixed step-size compared with the diminishing step-size of the corresponding subgradient counterpart.
The author of \cite{hong2017distributed} introduced an ADMM based incremental method for asynchronous non-convex optimization problems.
However, the results in \cite{bertsekas2011incremental,chazan1969chaotic,bertsekas1989parallel,zhou2018distributed,kibardin1980decomposition,nedich2001distributed,hong2017distributed,kumar2016asynchronous,tian2020achieving,hale2017asynchronous,cannelli2020asynchronous} are limited to either smooth individual objective functions or uncoupled constraints. In addition, the incremental (sub)gradient methods require certain fusion node to update the full system-wide variables continuously.

More detailed comparisons with the aforementioned works are listed as follows. (i) Different from \cite{xu2017convergence,wei20131,notarnicola2016asynchronous,bastianello2020asynchronous,mansoori2019fast}, in this work, the probability distribution of the asynchronous factors in the network is not required, which overcomes the previously discussed drawbacks of stochastic optimizations. (ii) In terms of the mathematical problem setup, the proposed Asyn-PPG algorithm can handle the non-smoothness of all the individual objective functions, which is not considered in \cite{xu2017convergence,chazan1969chaotic,bertsekas1989parallel,hong2017distributed,kumar2016asynchronous,mansoori2019fast,tian2020achieving,hale2017asynchronous}\footnote[2]{In \cite{hong2017distributed,kumar2016asynchronous}, the objective function of some distributed nodes is assumed to be smooth.}. The algorithms proposed in \cite{bertsekas2011incremental,chazan1969chaotic,kibardin1980decomposition,zhou2018distributed,nedich2001distributed,cannelli2020asynchronous} cannot address coupling constraints, and it is unclear whether the particularly concerned (consensus) constraints in \cite{notarnicola2016asynchronous,xu2017convergence,mansoori2019fast,bastianello2020asynchronous,hong2017distributed,tian2020achieving,bertsekas1989parallel,kumar2016asynchronous} can be extended to general linear equality/inequality constraints. In addition, the algorithms proposed in \cite{xu2017convergence,chazan1969chaotic,tian2020achieving,mansoori2019fast} cannot handle local constraints in their problems.

We hereby summarize the contributions of this work as follows.
\begin{itemize}
  \item A penalty based Asyn-PPG algorithm is proposed for solving a linearly constrained composite optimization problem in a partially asynchronous network. More precisely, we take the local/global constraints, non-smoothness of objective functions, asynchronous updates, and communication delays into account simultaneously with deterministic analysis, which, to the best knowledge of the authors, hasn't been addressed in the existing research works (e.g., \cite{xu2017convergence,wei20131,notarnicola2016asynchronous,bastianello2020asynchronous,mansoori2019fast,bertsekas2011incremental,chazan1969chaotic,kibardin1980decomposition,zhou2018distributed,nedich2001distributed,tian2020achieving,cannelli2020asynchronous,hong2017distributed,bertsekas1989parallel,kumar2016asynchronous,hale2017asynchronous}) and, hence, can adapt to more complicated optimization problems in asynchronous networks with deterministic convergence result.
  \item An SAN model is established by splitting the whole time domain into sequential time slots. In this model, all the agents are allowed to execute multiple updates asynchronously in each slot. Moreover, the agents only access the state of others at the beginning of each slot, which alleviates the intensive message exchanges in the network. In addition, the proposed interaction mechanism allows for communication delays among the agents, which are not considered in  \cite{bastianello2020asynchronous,xu2017convergence,wei20131,notarnicola2016asynchronous},
      and can also relieve the overload of certain central node as discussed in \cite{bertsekas2011incremental,nedich2001distributed,hong2017distributed,kibardin1980decomposition,hale2017asynchronous}.
  \item By the proposed Asyn-PPG algorithm, a periodic convergence rate $\mathcal{O}(\frac{1}{K})$ can be guaranteed with the coefficient of penalties synchronized at the end of each slot. The feasibility of the Asyn-PPG algorithm is verified by solving a distributed least absolute shrinkage and selection operator (LASSO) problem and a social welfare optimization problem in the electricity market.
\end{itemize}


The rest of this paper is organized as follows. {Section \ref{se2}} includes some frequently used notations and definitions in this work. {Section \ref{se3}} formulates the considered optimization problem. Basic definitions and assumptions of the SAN are provided therein. {Section \ref{se4}} presents the proposed Asyn-PPG algorithm and relevant propositions to be used in the subsequent analysis. In {Section \ref{se5}}, the main theorems on the convergence analysis of the Asyn-PPG algorithm are provided. {Section \ref{se6}} verifies the feasibility of the Asyn-PPG algorithm by two motivating applications. {Section \ref{se7}} concludes this paper.

\section{Preliminaries}\label{se2}
In the following, we present some preliminaries on notations, graph theory, and proximal mapping to be used throughout this work.

\subsection{Notations}

Let $\mid \mathcal{A}\mid$ be the size of set $\mathcal{A}$. $\mathbb{N}$ and $\mathbb{N}_+$ denote the non-negative integer space and positive integer space, respectively. $\mathbb{R}_{\succeq \mathbf{u}}^n$ denotes the $n$-dimensional Euclidian space with each element larger than or equal to the corresponding element in $\mathbf{u}$. $\parallel \cdot\parallel_1$ and $\parallel \cdot \parallel$ denote the $l_1$ and $l_2$-norms, respectively. $\langle \cdot, \cdot \rangle$ is an inner product operator. $\otimes$ is the Kronecker product operator. $\mathbf{0}$ and $\mathbf{1}$ denote the column vectors with all elements being 0 and 1, respectively. $\mathbf{I}_n$ and $\mathbf{O}_{m \times n}$ denote the $n$-dimensional identity matrix and $(m \times n)$-dimensional zero matrix, respectively. $\mathbf{relint}\mathcal{A}$ represents the relative interior of set $\mathcal{A}$.

\subsection{Graph Theory}
A multi-agent network can be described by an undirected graph ${\mathcal{G}}= \{{\mathcal{V}},{\mathcal{E}}\}$, which is composed of the set of vertices ${\mathcal{V}} = \{1,2,...,N \}$ and set of edges ${\mathcal{E}} \subseteq \{ (i,j)| i,j \in \mathcal{V} \hbox{ and } i \neq j\}$ with $(i,j) \in \mathcal{E}$ an unordered pair. A graph $\mathcal{G}$ is said connected if there exists at least one path between any two distinct vertices. A graph $\mathcal{G}$ is said fully connected if any two distinct vertices are connected by a unique edge. ${\mathcal{V}}_i =  \{ j | (i,j) \in \mathcal{E}\}$ denotes the set of the neighbours of agent $i$. Let ${\mathbf{L}} \in \mathbb{R}^{N \times N}$ denote the Laplacian matrix of ${\mathcal{G}}$. Let $l_{ij}$ be the element at the cross of the $i$th row and $j$th column of ${\mathbf{L}}$. Thus, $l_{ij} = -1$ if $(i,j) \in {\mathcal{E}}$, $l_{ii} = \mid {\mathcal{V}}_i \mid$, and $l_{ij} = 0$ otherwise, $i,j \in {\mathcal{V}}$ \cite{chung1997spectral}. 

\subsection{Proximal Mapping}
A proximal mapping of a closed, proper, convex function $\zeta: \mathbb{R}^n \rightarrow (-\infty,+\infty]$ is defined by
\begin{align}\label{d1}
\mathrm{prox}^a_{\zeta} (\mathbf{u})= \arg \min \limits_{\mathbf{v} \in \mathbb{R}^n} ( \zeta(\mathbf{v}) + \frac{1}{2a} \parallel \mathbf{v} - \mathbf{u}\parallel^2 ),
\end{align}
with step-size ${a}>0$ \cite{parikh2014proximal}.

\section{Problem Formulation and Network Modeling}\label{se3}

The considered mathematical problem and the proposed network model are presented in this section.

\subsection{The Optimization Problem}
In this paper, we consider a multi-agent network ${\mathcal{G}}=\{{\mathcal{V}},{\mathcal{E}}\}$. $f_i:\mathbb{R}^M \rightarrow (-\infty,+\infty]$ and $h_i: \mathbb{R}^M \rightarrow (-\infty,+\infty]$ are private objective functions of agent $i$, where $f_i$ is smooth and $h_i$ is possibly non-smooth, $i \in \mathcal{V}$. $\mathbf{x}_i= (x_{i1},...,x_{iM})^{\top} \in \mathbb{R}^M$ is the strategy vector of agent $i$, and $\mathbf{x}= (\mathbf{x}^{\top}_1,...,\mathbf{x}^{\top}_{N})^{\top}  \in \mathbb{R}^{M N}$ is the collection of all strategy vectors. A linearly constrained optimization problem of $\mathcal{V}$ can be formulated as
\begin{align}\label{}
\hbox{\textbf{(P1)}}: \quad
\min \limits_{\mathbf{x}}  \quad & {F}(\mathbf{x}) = \sum_{i \in \mathcal{V}} (f_i(\mathbf{x}_i) + h_i(\mathbf{x}_i)) \nonumber \\
\hbox{subject to} \quad  & \mathbf{A} \mathbf{x} = \mathbf{0}, \label{4}
\end{align}
where $\mathbf{A} \in \mathbb{R}^{B \times  NM}$.
For the convenience of the rest discussion, we define $f(\mathbf{x})= \sum_{i\in \mathcal{V}}f_i(\mathbf{x}_i)$, $h(\mathbf{x})= \sum_{i\in \mathcal{V}}h_i(\mathbf{x}_i)$, and $ {F}_i(\mathbf{x}_i) =   f_i(\mathbf{x}_i) + h_i(\mathbf{x}_i)$. Let $\mathbf{A}_i \in \mathbb{R}^{B \times M}$ be the $i$th column sub-block of $\mathbf{A}$, i.e., $\mathbf{A}=(\mathbf{A}_1,...,\mathbf{A}_i,...,\mathbf{A}_{N})$. Let $\mathbf{W} = \mathbf{A}^{\top} \mathbf{A} \in \mathbb{R}^{M N \times M{N}}$ and $\mathbf{W}_{ij}$ be the $(i,j)$th $(M \times M)$-dimensional sub-block of $\mathbf{W}$. Define $\mathbf{W}_i= (\mathbf{W}_{i1},...,\mathbf{W}_{iN}) = \mathbf{A}_i^{\top} \mathbf{A} \in \mathbb{R}^{M \times M{N}}$.

\begin{Assumption}\label{a0}
(Connectivity) $\mathcal{G}$ is undirected and fully connected.\footnote[3]{Strictly speaking, in this work, the requirement on the connectivity of the graph depends on how the individual variables are coupled in (\ref{4}). In some specific problems, $\mathcal{G}$ is not necessarily fully connected (see an example in Section \ref{rm2}).}
\end{Assumption}

\begin{Assumption}\label{a1}
({Convexity}) $f_i$ is proper, $L_i$-Lipschitz continuously differentiable and $\mu_i$-strongly convex, $L_i>0$, $\mu_i > 0$; $h_i$ is proper, convex and possibly non-smooth, $i\in \mathcal{V}$.
\end{Assumption}

The assumptions in Assumption \ref{a1} are widely used in composite optimization problems \cite{florea2020generalized,beck2014fast,notarnicola2016asynchronous}.

\begin{Assumption}\label{a1-1}
 (Constraint Qualification \cite{boyd2004convex}) There exists an $\breve{\mathbf{x}} \in \mathbf{relint} \mathcal{D}$ such that $\mathbf{A} \breve{\mathbf{x}} = \mathbf{0}$, where $\mathcal{D}$ is the domain of $F(\mathbf{x})$.
\end{Assumption}

\begin{Remark}\label{re12}
Problem (P1) defines a prototype of a class of optimization problems. One may consider an optimization problem with local convex constraint $\mathbf{x}_i \in \Omega_i$ and coupling inequality constraint $\mathbf{A} \mathbf{x} + \mathbf{b} \preceq \mathbf{0}$ by introducing slack variables and indicator functions into Problem (P1) \cite{boyd2004convex}, which gives
\begin{align}\label{}
\hbox{\textbf{(P1+)}}: \min \limits_{\mathbf{x}_i, \mathbf{y}, \forall i \in \mathcal{V}}  \quad & \sum_{i \in \mathcal{V}} (f_i(\mathbf{x}_i) + h_i(\mathbf{x}_i)  + \mathbb{I}_{\Omega_i} (\mathbf{x}_{i}))  + \mathbb{I}_{\mathbb{R}^B_{\succeq \mathbf{b}}} (\mathbf{y}) \nonumber \\
\hbox{subject to} \quad & \mathbf{A} \mathbf{x} + \mathbf{y} = \mathbf{0},
\end{align}
where $\Omega_i \subseteq \mathbb{R}^{M}$ is non-empty, convex and closed, $\mathbf{y} \in \mathbb{R}^B$ is a slack variable, and
\begin{align}\label{}
& \mathbb{I}_{\Omega_i} (\mathbf{x}_{i}) = \left\{
\begin{array}{cc}
  0 & \mathbf{x}_{i} \in \Omega_i, \\
  + \infty &  \hbox{otherwise,}
\end{array}  \right.   \\
& \mathbb{I}_{\mathbb{R}^B_{\succeq \mathbf{b}}} (\mathbf{y}) = \left\{
\begin{array}{cc}
  0 & \mathbf{y} \in \mathbb{R}^B_{\succeq \mathbf{b}}, \\
  + \infty & \hbox{otherwise.}
\end{array}  \right.
\end{align}
To realize decentralized computations, $\mathbf{y}$ can be decomposed and assigned to each of the agents. Since $\mathbb{I}_{\Omega_i}$ and $\mathbb{I}_{\mathbb{R}^B_{\succeq \mathbf{b}}}$ are proper and convex, the structure of Problem (P1+) is consistent with that of Problem (P1).
\end{Remark}


\subsection{Characterization of Optimal Solution}

By recalling Problem (P1), we define Lagrangian function
\begin{align}
{\mathcal{L}}({\mathbf{x}}, \bm{\lambda})=  & {F}(\mathbf{x})  + \langle \bm{\lambda}, \mathbf{A}\mathbf{x} \rangle,
\end{align}
where $\bm{\lambda} \in \mathbb{R}^{B}$ is the Lagrangian multiplier vector. Let $\mathcal{X}$ be the set of the saddle points of ${\mathcal{L}}({\mathbf{x}}, \bm{\lambda})$. Then, any saddle point $({\mathbf{x}}^*,{\bm{\lambda}}^*) \in \mathcal{X}$ can be characterized by \cite{boyd2004convex}
\begin{align}\label{5}
({\mathbf{x}}^*,{\bm{\lambda}}^*) = \arg \max_{\bm{\lambda}} \min_{\mathbf{{x}}} {\mathcal{L}}({\mathbf{x}},\bm{\lambda}),
\end{align}
where ${\mathbf{x}}^*=((\mathbf{x}^*_1)^{\top},...,(\mathbf{x}^*_{N})^{\top})^{\top}$ and $\bm{\lambda}^* = (\lambda^*_1,...,\lambda^*_B)^{\top}$. Then, $\forall \mathbf{x} \in \mathbb{R}^{M{N}}$, we have
\begin{align}\label{}
  & {F}(\mathbf{x}) + \langle \bm{\lambda}^*, \mathbf{A}\mathbf{x}\rangle  - {F}(\mathbf{x}^*) - \langle \bm{\lambda}^*, \mathbf{A}\mathbf{x}^* \rangle  \geq 0. \nonumber
\end{align}
With the fact $ \mathbf{A}\mathbf{x}^*  = \mathbf{0}$, we can obtain
\begin{align}\label{sd1}
  & {F}(\mathbf{x}) + \langle \bm{\lambda}^*, \mathbf{A}\mathbf{x}\rangle  - {F}(\mathbf{x}^*) \geq 0.
\end{align}

\subsection{Slot-based Asynchronous Network}

Regarding the asynchrony issues outlined in Section \ref{bnm}, we propose an SAN model which consists of the following two key features.
\begin{enumerate}
  \item  The whole time domain is split into sequential time slots and the agents are permitted to execute multiple updates in each slot. There is no restriction on which time instant should be taken, which enables the agents to act asynchronously.
  \item All the agents can access the information of others in the previous slot at the beginning of the current slot, but the accessed state information may not be the latest depending on how large the communication delay of the network is.
\end{enumerate}
For practical implementation, the proposed SAN model is promising to be applied in some time-slot based problems, such as bidding and auctions in the electricity market and task scheduling problems in multi-processor systems \cite{david2000strategic,andersson2010implementing}.

The detailed mathematical descriptions of SAN are presented as follows. We let $\mathcal{T}= \{0,1,2,...\}$ be the collection of the whole discrete-time instants and $\mathcal{M}= \{t_m \}_{m\in \mathbb{N}} \subseteq \mathcal{T}$ be the sequence of the boundary of successive time slots. $\mathcal{T}_i \subseteq \mathcal{T}$ is the action clock of agent $i$. Slot $m$ is defined as the time interval $[t_m,t_{m+1})$.

\begin{Assumption}\label{a4-1}
(Uniform Slot Width) The width of slots is uniformly set as $H$, i.e., $t_{m+1} - t_{m} = H$, $H \in \mathbb{N}_+$, $m\in \mathbb{N}$.
\end{Assumption}

\begin{Assumption}\label{a4}
(Frequent Update) Each agent performs at least one update within $[t_m,t_{m+1})$, i.e., $\mathcal{T}_i \cap [t_m,t_{m+1}) \neq \emptyset$, $\forall i \in \mathcal{V}$, $m \in \mathbb{N}$.
\end{Assumption}
The update frequency of agent $i$ in slot $m$ is defined by $P_{i,m}$, i.e., $P_{i,m} = \mid \mathcal{T}_i \cap [t_m,t_{m+1}) \mid \in [1,H]$. Define $\mathcal{P}^m_i=  \{1,2,...,P_{i,m} \}$, $i\in \mathcal{V}$, $m \in \mathbb{N}$. Let $t_m^{(n)} \in \mathcal{T}$ denote the instant of the $n$th update in slot $m$. For the mathematical derivation purpose, we let
\begin{align}
 & t_{m}^{(P_{i,m}+1)}  = t_{m+1}^{(1)}, \label{pan2} \\
 & t_{m+1}^{(0)}  = t_{m}^{(P_{i,m})}. \label{pan3}
\end{align}
(\ref{pan2}) and (\ref{pan3}) are the direct extensions of the action indexes between two sequential slots. That is, the $1$st action instant in slot $m+1$ is equivalent to the $(P_{i,m}+1)$th action instant in slot $m$; the $0$th action instant in slot $m+1$ is equivalent to the $P_{i,m}$th action instant in slot $m$.

\begin{Proposition}\label{pp10}
In the proposed SAN, $\forall i \in \mathcal{V}$, $m \in \mathbb{N}$, we have the following inequality:
\begin{align}
 & t_{m}^{(P_{i,m})} \leq t_{m+1}-1 < t_{m+1} \leq t_{m+1}^{(1)}. \label{pan1}
\end{align}
\end{Proposition}
\begin{proof}
Note that $t_{m}^{(P_{i,m})}$ and $t_{m+1}^{(1)}$ are the last update instant in $[t_m,t_{m+1})$ and the first update instant in $[t_{m+1},t_{m+2})$ of agent $i$, respectively. Therefore, the validation of (\ref{pan1}) is straightforward.
\end{proof}

\begin{Assumption}\label{as1}
(Information Exchange) Each agent always knows the latest information of itself, but the state information of others can only be accessed at the beginning of each slot, i.e., $t_m$, $\forall m \in \mathbb{N}$.
\end{Assumption}

Assumption \ref{as1} enables the agents to only communicate at the instants in $\mathcal{M}$, which can relieve the intensive information exchanges in the network. However, due to communication delays, in slot $m$, certain agent $i$ may not access the latest information of agent $j$ at time $t_m$, i.e., $\mathbf{x}_j(t_m)$, $j \in \mathcal{V} \setminus \{i\}$, but a possibly delayed version $\mathbf{x}_j(\tau(t_m))$ with $\tau(t_m) \leq t_m$, $\tau(t_m) \in \mathcal{T}$. $\mathbf{x}_j(\tau(t_m)) \neq \mathbf{x}_j(t_m)$ means that agent $j$ performs update(s) within $[\tau(t_m),t_m)$. Therefore, the full state information available at instant $t_{m}$ may not be $\mathbf{x}(t_m)$ but a delayed version $\mathbf{x}^{\mathrm{d}}(t_m)= ((\mathbf{x}_1^{\mathrm{d}})^{\top}(t_m),...,(\mathbf{x}_{N}^{\mathrm{d}})^{\top}(t_m))^{\top} = (\mathbf{x}^{\top}_1(\tau (t_m)),..., \mathbf{x}^{\top}_{N}(\tau (t_m)))^{\top} \in \mathbb{R}^{M{N}}$.\footnote[4]{In slot $m$, the time instant of the historical state, i.e., $\tau(t_{m})$, is identical, which means the communication delay is uniform for the all the agents (as discussed in \cite{wang2013consensus}) in certain slot and can be varying in different slots.}
\begin{Assumption}\label{a5}
({Bounded Delay}) The communication delays in the network are upper bounded by $D \in \mathbb{N}_+$ with $D \leq H$, i.e., $t_m- \tau(t_m) \leq D$, $\forall m \in \mathbb{N}$.
\end{Assumption}


In slot $m$, the historical state of agent $i$ can be alternatively defined by $ \mathbf{x}_i(t^{(n_{i,m})}_m) =  \mathbf{x}_i^{\mathrm{d}}(t_{m+1})$, where $t^{(n_{i,m})}_m$ is the largest integer no greater than $\tau(t_{m+1})$ in set $\mathcal{T}_i$, and $n_{i,m} \in \mathbb{N}_+$ is the index of the update. Then, the number of updates within $[t^{(n_{i,m})}_m, t^{(P_{i,m})}_m]$ should be no greater than the number of instants in $[\tau(t_{m+1}), t_{m+1})$, i.e.,
\begin{equation}\label{6}
P_{i,m} - n_{i,m} \leq t_{m+1} -1 - \tau(t_{m+1}) \leq D -1.
\end{equation}
The relationship among $\mathcal{T}$, $\mathcal{T}_i$ and delay in slot $m$ is illustrated in Fig. \ref{1234}.
\begin{figure}[htbp]
  \centering
  \includegraphics[width=7cm]{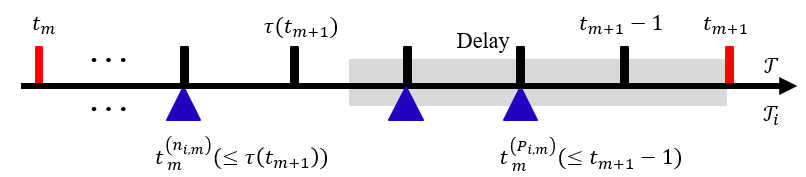}\\
  \caption{An illustration of the relationship among $\mathcal{T}$, $\mathcal{T}_i$ and delay in slot $m$. In this example, $P_{i,m} - n_{i,m}=2$ and $t_{m+1} - \tau(t_{m+1})=4$, which satisfies (\ref{6}).}\label{1234}
\end{figure}

\section{Asynchronous Penalized Proximal Gradient Algorithm}\label{se4}

Based on the SAN model, the Asyn-PPG algorithm is designed in this section.

Let $\{\alpha_i(t^{(n)}_m)_{>0}\}_{n \in \mathcal{P}^m_i}$ and $\{\eta_i(t^{(n)}_m)_{>0}\}_{n \in \mathcal{P}^m_i}$ be two sequences assigned to agent $i$ in slot $m$. In addition, we introduce a sequence $\{\alpha_i(t_{m+1}-1) \}_{m \in \mathbb{N}}$ and a scalar $\beta > 0$, where $\alpha_i(t_{m+1}-1)$ is the value of $\alpha_i$ at time instant $t_{m+1}-1$. Then, by considering the overall action/non-action instants, the updating law of the agents is given in Algorithm \ref{a1x}.\footnote[5]{We assume that the Asyn-PPG algorithm starts from slot $1$ by viewing the states in slot $0$ as historical data.}
\begin{algorithm}
\caption{Asynchronous Penalized Proximal Gradient Algorithm}\label{a1x}
\begin{algorithmic}[1]
\State Initialize $\mathbf{x}_i(t^{(1)}_1)$, $\mathbf{x}^{\mathrm{d}}(t_1)$, $\forall i \in \mathcal{V}$.
\State For all $t  \in  \mathcal{T}, i \in \mathcal{V}, n \in \mathcal{P}^m_i, m \in \mathbb{N}_+$,
\State if $t \in  \mathcal{T}_i\cap [t_m,t_{m+1})$, then
\State $\quad t_m^{(n)} \leftarrow t$,
\State $\quad$update parameters: $\alpha_i(t^{(n)}_m)$, $\alpha_i(t_{m+1}-1)$, $\eta_i(t^{(n)}_m)$,
\State $\quad$update state:
\begin{align}\label{}
        & \mathbf{x}_i(t^{(n)}_{m}+1) =  \mathrm{prox}_{h_i}^{\eta_i(t^{(n)}_m)} (\mathbf{x}_i(t^{(n)}_m) - \eta_i(t^{(n)}_m) \nonumber \\
        & \quad \quad \cdot (\nabla f_i(\mathbf{x}_i(t^{(n)}_m)) + \frac{\beta \mathbf{W}_i}{\alpha_i (t_{m+1}-1)}\mathbf{x}^{\mathrm{d}}(t_m))); \nonumber
\end{align}
\State if $ t \in [t_m,t_{m+1}) \hbox{ \& }t \notin  \mathcal{T}_i$, then
\State $\quad$ $\mathbf{x}_i(t+1) = \mathbf{x}_i(t)$.
\State Stop under certain convergence criterion.
\end{algorithmic}
\end{algorithm}
Note that $\mathbf{W}_i\mathbf{x}^{\mathrm{d}}(t_m) = \mathbf{A}^{\top}_i \mathbf{A} \mathbf{x}^{\mathrm{d}}(t_m)$. Hence, $\frac{\beta \mathbf{W}_i\mathbf{x}^{\mathrm{d}}(t_m)}{\alpha_i (t_{m+1}-1)}$ can be viewed as a violation penalty of a ``delayed'' global constraint $\mathbf{A} \mathbf{x}^{\mathrm{d}} (t_m)= \mathbf{0}$, which is a variant of the penalty method studied in \cite{li2017convergence}.


Algorithm \ref{a1x} provides a basic framework for solving the proposed optimization problem in the SAN. An illustrative state updating process by Asyn-PPG algorithm in a 3-agent SAN is shown in Fig. \ref{p3}.
\begin{figure}[htbp]
  \centering
  \includegraphics[height=7cm,width=10cm]{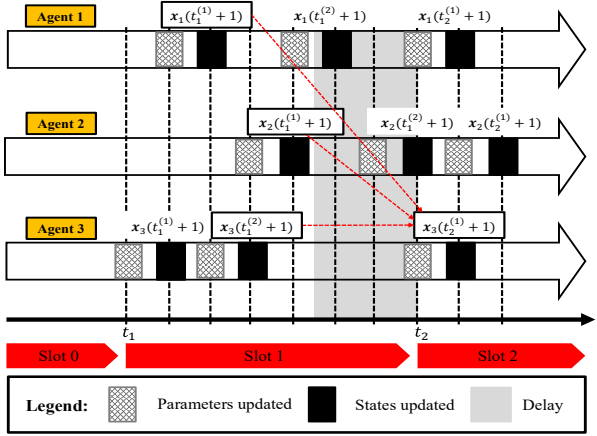}\\
  \caption{An illustrative updating process of the Asyn-PPG algorithm in a 3-agent SAN. In this example, the state of the agents evolves from $(\mathbf{x}_1(t^{(1)}_1),\mathbf{x}_2(t^{(1)}_1),\mathbf{x}_3(t^{(1)}_1))$ to $(\mathbf{x}_1(t^{(1)}_2 +1 ),\mathbf{x}_2(t^{(1)}_2+1),\mathbf{x}_3(t^{(1)}_2+1))$ with the historical state provided at the beginning of each time slot. This updating process is parallel but asynchronous due to the arbitrarily determined action instants of the agents. Specifically, to compute $\mathbf{x}_3(t^{(1)}_2+1)$, the state information available for agent 3 in slot 2 is $(\mathbf{x}_1(t^{(1)}_1+1),\mathbf{x}_2(t^{(1)}_1+1),\mathbf{x}_3(t^{(2)}_1+1))$ rather than $(\mathbf{x}_1(t^{(2)}_1+1),\mathbf{x}_2(t^{(2)}_1+1),\mathbf{x}_3(t^{(2)}_1+1))$ since the action instants of $\mathbf{x}_1(t^{(1)}_1+1) \rightarrow \mathbf{x}_1(t^{(2)}_1+1)$ and $\mathbf{x}_2(t^{(1)}_1+1)  \rightarrow \mathbf{x}_2(t^{(2)}_1+1)$ are too close to $t_2$, and therefore, $\mathbf{x}_1(t^{(2)}_1+1)$ and $\mathbf{x}_2(t^{(2)}_1+1)$ can not reach agent 3 by $t_2$ due to the communication delays in the network.}\label{p3}
\end{figure}

Based on Asyn-PPG algorithm, we have the following two propositions.
\begin{Proposition}\label{pp1}
({Equivalent Representation A}) By Algorithm \ref{a1x}, $\forall i \in \mathcal{V}$, $n \in \mathcal{P}^m_i$, $m \in \mathbb{N}$, we have
\begin{subequations}
\begin{align}\label{}
& \mathbf{x}_i (t^{(1)}_{m+1}) =  \mathbf{x}_i (t^{(P_{i,m}+1)}_{m}), \label{e1} \\
 & \alpha_i(t^{(P_{i,m})}_{m})=\alpha_i(t^{(0)}_{m+1}), \label{e6} \\
  & \eta_i(t^{(P_{i,m})}_{m})= \eta_i(t^{(0)}_{m+1}), \label{e7} \\
& \mathbf{x}_i (t^{(n)}_{m}+1) =  \mathbf{x}_i (t^{(n+1)}_{m}), \label{e0} \\
& \mathbf{x}_i(t_{m+1}) = \mathbf{x}_i (t^{(P_{i,m}+1)}_{m}), \label{e2} \\
&  \alpha_i(t_{m}^{(P_{i,m})}) =  \alpha_i(t_{m}^{(P_{i,m}+1)} -1), \label{e4} \\
&  \alpha_i(t_{m+1}-1)=  \alpha_i(t_{m}^{(P_{i,m}+1)} -1), \label{e4+1} \\
 & \eta_i(t_{m}^{(P_{i,m})}) = \eta_i(t_{m}^{(P_{i,m}+1)} -1), \label{e5} \\
  & \eta_i(t_{m+1}-1) = \eta_i(t_{m}^{(P_{i,m}+1)} -1). \label{e5+1}
\end{align}
\end{subequations}
\end{Proposition}

\begin{proof}
See Appendix \ref{pp1p}.
\end{proof}

\begin{Proposition}\label{la2}
By Algorithm \ref{a1x}, $\forall m \in \mathbb{N}$, we have
\begin{align}
 & \parallel \mathbf{x} (t_{m+1}) -  \mathbf{x}^{\mathrm{d}}(t_{m+1}) \parallel^2 \nonumber \\
  & \leq  \sum_{i \in \mathcal{V}}\sum_{n = 1}^{P_{i,m}} D \parallel \mathbf{x}_i(t^{(n+1)}_{m}) - \mathbf{x}_i(t^{(n)}_{m}) \parallel^2, \label{l1}\\
  & \parallel \mathbf{x} (t_{m+1}) -  \mathbf{x}(t_m) \parallel^2 \nonumber \\
  & \leq  \sum_{i \in \mathcal{V}}\sum_{n = 1}^{P_{i,m}} H \parallel \mathbf{x}_i(t^{(n+1)}_{m}) - \mathbf{x}_i(t^{(n)}_{m}) \parallel^2. \label{l2}
\end{align}
\end{Proposition}

\begin{proof}
See Appendix \ref{la2p}.
\end{proof}

\section{Main Result}\label{se5}

In this section, we will establish the parameters of the Asyn-PPG algorithm for solving Problem (P1) in the SAN.

\subsection{Determination of Parameters}

In Algorithm \ref{a1x}, the penalty coefficient $\frac{\beta}{\alpha_i (t_{m+1}-1)}$ is designed to be increased steadily with $m \rightarrow + \infty$, which can speed up convergence rate compared with the corresponding fixed penalty method. The updating law of sequence $\{\alpha_i(t)_{>0}\}_{t\in \mathcal{T}_i}$ for agent $i$ is designed as
\begin{equation}\label{th12+1}
  \frac{1- {\theta_i(t^{(n)}_m)}}{\alpha_i(t^{(n)}_m)} = \frac{1}{\alpha_i(t^{(n-1)}_m)},
\end{equation}
and sequence $\{\eta_i(t)_{>0}\}_{t\in \mathcal{T}_i}$ is decided by
\begin{equation}\label{th14+1}
\frac{\theta_i(t^{(n)}_m) - \theta_i(t^{(n)}_m)\eta_i(t^{(n)}_m)\mu_i}{\eta_i(t^{(n)}_m)\alpha_i(t^{(n)}_m)} \leq \frac{\theta_i(t^{(n-1)}_{m})}{\eta_i(t^{(n-1)}_{m})\alpha_i(t^{(n-1)}_{m})},
\end{equation}
with $\theta_i(t^{(n)}_m) \in (0,1)$, $\forall i \in \mathcal{V}$, $m \in \mathbb{N}_+$, $n \in \mathcal{P}^m_i$.
\begin{Proposition}\label{pr1}
({Strictly Decreasing}) Given that $\alpha_i(t^{(0)}_1) > 0$, the sequence $\{ \alpha_i(t)\}_{t \in \mathcal{T}_i}$ generated by (\ref{th12+1}) is strictly decreasing with $t \rightarrow + \infty$, $\forall i \in \mathcal{V}$.
\end{Proposition}
\begin{proof}
The validation of proposition \ref{pr1} is straightforward with $\theta_i(t^{(n)}_m) \in (0,1)$ in (\ref{th12+1}) and relation (\ref{e6}).
\end{proof}
\begin{Proposition}\label{pp2}
({Equivalent Representation B}) By Algorithm \ref{a1x} and (\ref{th12+1}), $\forall i \in \mathcal{V}$, $m \in \mathbb{N}$, we have
\begin{subequations}
\begin{align}\label{}
  & \theta_i(t^{(P_{i,m})}_{m})= \theta_i(t^{(0)}_{m+1}), \label{e9} \\
 & \theta_i(t_{m}^{(P_{i,m})}) = \theta_i(t_{m}^{(P_{i,m}+1)} -1), \label{e8} \\
 & \theta_i(t_{m+1}-1)= \theta_i(t_{m}^{(P_{i,m}+1)} -1). \label{e8+1}
\end{align}
\end{subequations}
\end{Proposition}
\begin{proof}
Note that by (\ref{th12+1}), the values of $\theta_i$ and $\alpha_i$ are updated simultaneously at any instant in $\mathcal{T}_i$ after the initialization of $\alpha_i$. Hence, by recalling equivalent representations (\ref{e6}), (\ref{e4}) and (\ref{e4+1}), (\ref{e9})-(\ref{e8+1}) can be verified.
\end{proof}


Before more detailed discussions on the updating law of $\alpha_i$ and $\theta_i$, we introduce the following definition.

\begin{Definition}\label{dd1}
(Synchronization of $\{\alpha_i(t_{m}-1) \}_{m \in \mathbb{N}_+}$) In the SAN, sequence $\{\alpha_i(t_{m}-1) \}_{m \in \mathbb{N}_+}$ is synchronized if
\begin{align}\label{tha1}
& \alpha_1(t_{m}-1)= ... = \alpha_i(t_{m}-1)= ... = \alpha_{N}(t_{m}-1).
\end{align}
\end{Definition}

Under condition (\ref{tha1}), we further define a common sequence $\{\alpha(t_{m}-1) \}_{m \in \mathbb{N}_+}$ with
\begin{align}\label{cc}
\alpha(t_{m}-1) =\alpha_i(t_{m}-1)
\end{align}
for convenience purpose, $\forall i \in \mathcal{V}$.

The synchronization strategy for $\{\alpha_i(t_{m}-1) \}_{m \in \mathbb{N}_+}$ is not unique. One feasible realization is provided as follows.

\begin{Lemma}\label{lm3}
Let (\ref{th12+1}) hold. Let
\begin{align}\label{1}
{\alpha_1(t^{(P_{1,0})}_{0})}= ... = {\alpha_{N}(t^{(P_{{N},0})}_{0})}
\end{align}
and
\begin{align}\label{lp1}
\frac{\theta_i(t^{(n)}_m)}{\alpha_i(t^{(n)}_m)} = \frac{1}{P_{i,m}},
\end{align}
$\forall i \in \mathcal{V}$, $m \in \mathbb{N}_+$, $n \in \mathcal{P}^m_i$. Then, we have (\ref{tha1}),
\begin{align}
& \alpha(t_m-1)=\frac{\alpha(t_{1}-1)}{(m-1)\alpha(t_{1}-1)+1}, \label{17-1} \\
& \frac{1}{\alpha_i(t_m^{(n)})} = \frac{n}{P_{i,m}} + \frac{1}{\alpha(t_1 - 1)} + m-1, \label{17-2} \\
& \frac{\alpha_i(t_m^{(n)})}{\alpha(t_{m+2}-1)} \in ( 1,\Pi], \quad \Pi =  \frac{2\alpha(t_1-1) + 1}{\frac{1}{H}\alpha(t_1-1) + 1}. \label{17-5}
\end{align}
\end{Lemma}

\begin{proof}
See Appendix \ref{lm3p}.
\end{proof}

\begin{Lemma}\label{lpl}
Let (\ref{lp1}) hold. Let
\begin{align}\label{llp}
& \frac{\eta_i(t_{m}-1)}{\eta_j(t_{m}-1)} = \frac{P_{j,m-1}}{P_{i,m-1}},
\end{align}
$\forall i,j \in \mathcal{V}$, $m \in \mathbb{N}_+$. Then,
\begin{align}\label{cx}
\frac{\theta_1(t_{m}-1)} {\alpha_1(t_{m}-1)\eta_1(t_{m}-1)} & = ... = \frac{\theta_i(t_{m}-1)} {\alpha_i(t_{m}-1)\eta_i(t_{m}-1)} \nonumber \\
= ... = & \frac{\theta_{N}(t_{m}-1)} {\alpha_{N}(t_{m}-1)\eta_{N}(t_{m}-1)}.
\end{align}
\end{Lemma}

\begin{proof}
See Appendix \ref{lplp}.
\end{proof}

Under condition (\ref{cx}), we define a common sequence $\{ \Xi_m \}_{m \in \mathbb{N}_+}$ with
\begin{align}\label{271}
\Xi_m = \frac{\theta_i(t_{m}-1)} {\alpha_i(t_{m}-1)\eta_i(t_{m}-1)}
\end{align}
for convenience purpose, $\forall i \in \mathcal{V}$.

\begin{Remark}
{Lemmas \ref{lm3}} and {\ref{lpl}} imply that the determination of $\alpha_i$, $\theta_i$, and $\eta_i$ requires some slot-wide knowledge of the actions, i.e., $P_{i,m}$, which is realizable when each agent knows the update frequency of itself.
\end{Remark}

\subsection{Convergence Analysis}

Based on the previous discussions, we are ready to provide the main results as follows.

\begin{Lemma}\label{z}
In the proposed SAN, suppose that {Assumptions \ref{a0}} to {\ref{a5}}, (\ref{th12+1}), and (\ref{th14+1}) hold. Then, by Algorithm \ref{a1x}, for any $({\mathbf{x}}^*,{\bm{\lambda}}^*) \in \mathcal{X}$, $i\in \mathcal{V}$, $m \in \mathbb{N}_+$, $n \in \mathcal{P}^m_i$, we have
\begin{align}\label{zt2}
 &\frac{1}{\alpha_i(t_{m+1}-1)}({F}_i(\mathbf{x}_i(t_{m+1})) - {F}_i(\mathbf{x}_i^*) + \langle \bm{\lambda}^*, \mathbf{A}_i\mathbf{x}_i(t_{m+1})\rangle) \nonumber \\
& - \frac{1}{\alpha_i(t_{m}-1)}({F}_i(\mathbf{x}_i(t_{m})) - {F}_i(\mathbf{x}_i^*) + \langle \bm{\lambda}^*, \mathbf{A}_i\mathbf{x}_i(t_{m})\rangle) \nonumber \\
   \leq & \frac{1}{\beta} \langle \bm{\lambda}^*- \frac{\beta \mathbf{A}\mathbf{x}^{\mathrm{d}}(t_m)}{\alpha_i(t_{m+1}-1)}, \frac{\beta\mathbf{A}_i\mathbf{x}_i(t_{m+1})}{\alpha_i(t_{m+1}-1)} - \frac{\beta\mathbf{A}_i\mathbf{x}_i(t_{m})}{\alpha_i(t_m-1)} \rangle  \nonumber \\
& + \sum_{n=1}^{P_{i,m}} \frac{1}{2\alpha_i(t^{(n)}_m)}(L_i- \frac{2-\theta_i(t^{(n)}_m)}{\eta_i(t^{(n)}_m)} ) \parallel \mathbf{x}_i(t_m^{(n+1)}) \nonumber \\
& -\mathbf{x}_i(t_m^{(n)}) \parallel^2 + \sum_{n=1}^{P_{i,m}} \frac{\theta_i(t_m^{(n)})}{\alpha_i(t_m^{(n)})} \langle \frac{\beta \mathbf{A}\mathbf{x}^{\mathrm{d}}(t_m)}{\alpha_i(t_{m+1}-1)},  \mathbf{A}_i \mathbf{x}_i^* \rangle \nonumber \\
& + \frac{\theta_i(t_{m}-1)}{2\alpha_i(t_{m}-1)\eta_i(t_{m}-1)} \parallel \mathbf{x}_i^*-\mathbf{x}_i(t_{m}) \parallel^2 \nonumber \\
&  - \frac{\theta_i(t_{m+1}-1)}{2\alpha_i(t_{m+1}-1)\eta_i(t_{m+1}-1)} \parallel \mathbf{x}_i^*-\mathbf{x}_i(t_{m+1}) \parallel^2.
\end{align}
\end{Lemma}

\begin{proof}
See Appendix \ref{z1}.
\end{proof}

{Lemma \ref{z}} provides a basic result for further convergence analysis. It can be seen that, in the proposed SAN, the state of agent $i$ is decided by its own parameters $\alpha_i$, $\theta_i$ and $\eta_i$, which are further decided by the action instants in $\mathcal{T}_i$. By the parameter settings in Lemmas \ref{lm3} and \ref{lpl}, we have the following theorem.

\begin{Theorem}\label{th1}
In the proposed SAN, suppose that {Assumptions \ref{a0}} to {\ref{a5}}, (\ref{th12+1}), (\ref{1}), and (\ref{lp1}) hold. Choose an $\eta_i(t^{(n)}_m)$ such that (\ref{th14+1}), (\ref{llp}), and
\begin{align}\label{th13}
& \frac{1}{\eta_i(t^{(n)}_m)} \geq  L_i + \frac{2  (H+D)\beta \Pi \parallel \mathbf{A} \parallel^2}{\alpha(t_{m+2}-1)}
\end{align}
hold, $\forall i \in \mathcal{V}$, $ m \in \mathbb{N}_+$, $n \in \mathcal{P}^m_i$. Then, by Algorithm \ref{a1x}, for certain $K \in \mathbb{N}_+$ and any $({\mathbf{x}}^*,{\bm{\lambda}}^*) \in \mathcal{X}$, we have
\begin{align}\label{}
& \mid  F(\mathbf{x}(t_{K+1}))  - F(\mathbf{x}^*) \mid \nonumber \\
 & \quad \quad \quad \leq  ( \Delta_1 + \Delta_2 \parallel \bm{\lambda}^* \parallel) \alpha(t_{K+1}-1), \label{f2} \\
  & \parallel \mathbf{A}\mathbf{x}  (t_{K+1}) \parallel \leq \Delta_2 \alpha (t_{K+1}-1), \label{f3}
\end{align}
where
\begin{align}\label{}
 \Delta_1= &  \frac{1}{\alpha(t_{1}-1)}({F}(\mathbf{x}(t_{1})) - {F}(\mathbf{x}^*) + \langle \bm{\lambda}^*, \mathbf{A}\mathbf{x}(t_{1})\rangle) \nonumber \\
 &  + \frac{1}{2\beta } \parallel \frac{\beta\mathbf{A}\mathbf{x}(t_1)}{\alpha(t_1-1)} - \bm{\lambda}^* \parallel^2 + \frac{\Xi_1}{2} \parallel \mathbf{x}^*-\mathbf{x}(t_1) \parallel^2 \nonumber \\
&  +  \sum_{i\in \mathcal{V}}\sum_{n=1}^{P_{i,0}} \frac{D\beta \parallel \mathbf{A} \parallel^2}{\alpha^2(t_{2}-1)} \parallel \mathbf{x}_i(t_{0}^{(n+1)}) - \mathbf{x}_i(t_{0}^{(n)}) \parallel^2, \label{c1}  \\
 \Delta_2= & \frac{\sqrt{2\beta \Delta_1} + \parallel \bm{\lambda}^* \parallel }{\beta}. \label{c2}
\end{align}
\end{Theorem}

\begin{proof}
See Appendix \ref{tp1}.
\end{proof}

\begin{Remark}\label{re1}
{Theorem \ref{th1}} provides a sufficient condition of the convergence of the Asyn-PPG algorithm, which is characterized by the initial state of all the time slots and results in a periodic convergence result with period length $H$ (see some similar periodic convergence results in \cite{tseng1991rate,zhou2018distributed}).
\end{Remark}

To achieve the result of {Theorem \ref{th1}}, we need to choose a suitable $\eta_i(t^{(n)}_m)$ which is located in the space determined by (\ref{th14+1}), (\ref{llp}) and (\ref{th13}) adaptively. In the following, we investigate the step-size $\eta_i(t^{(n)}_m)$ in the form of
\begin{align}\label{16}
& \frac{1}{\eta_i(t^{(n)}_m)} = P _{i,m} ( Q_{m} + \frac{2(H+D)\beta \Pi \parallel \mathbf{A} \parallel^2}{\alpha_i(t_{m+2}-1)} ),
\end{align}
where $Q_{m}$, $\alpha_i(t_{m+2}-1)$ and $\beta$ are to be determined ($P_{i,0}$ and $Q_0$ are defined to initialize $\eta_i$).

\begin{Lemma}\label{l5}
Suppose that (\ref{th12+1}), (\ref{1}), and (\ref{lp1}) hold. Let the step-size be in the form of (\ref{16}) and $Q_{m} \geq L^{\mathrm{g}}$ with $L^{\mathrm{g}} = \max_{j \in \mathcal{V}} L_j$. Then, (\ref{llp}) and (\ref{th13}) hold. In addition, $\forall i \in \mathcal{V}$, $m \in \mathbb{N}_+$, $n \in \mathcal{P}^m_i$, we have
\begin{align}\label{151}
 & \frac{\theta_i(t^{(n)}_m)}{\alpha_i(t^{(n)}_m)\eta_i(t^{(n)}_m)} - \frac{\theta_i(t^{(n-1)}_{m})} {\alpha_i(t^{(n-1)}_{m})\eta_i(t^{(n-1)}_{m})} \nonumber \\
  &  \leq  \max \{ 0,Q_m -  Q_{m-1} \nonumber \\
  &   \quad \quad \quad \quad \quad \quad + 2(H+D)\beta \Pi \parallel \mathbf{A} \parallel^2\}.
\end{align}
\end{Lemma}

\begin{proof}
See Appendix \ref{l5p}.
\end{proof}


\begin{Theorem}\label{la1}
In the proposed SAN, suppose that {Assumptions \ref{a0}} to {\ref{a5}}, (\ref{th12+1}), (\ref{1}), and (\ref{lp1}) hold and $\parallel \mathbf{A}  \parallel \neq 0$. The step-size is selected based on (\ref{16}). Then, by Algorithm \ref{a1x}, given that

1) there exist a $K \in \mathbb{N}_+$ and an $\epsilon>0$, such that
\begin{align}\label{27}
K \geq \frac{1}{\epsilon} - \frac{1}{\alpha(t_{1}-1)},
\end{align}

2) there exists a $Q_{m}$, such that
\begin{align}
& Q_{m} \geq L^{\mathrm{g}}, \label{33}\\
& Q_{m} - Q_{m-1} <  \frac{\mu}{H}, \label{35}
\end{align}
with $\mu=\min_{j \in \mathcal{V}} \mu_j$, and

3) $\beta$ is chosen as
\begin{align}\label{32-1}
 \beta \in (0, & \frac{ \frac{\mu}{H} - \max_{k \in \mathbb{N}_+} (Q_{k} - Q_{k-1})}{2(H+D) \Pi \parallel \mathbf{A} \parallel^2}],
\end{align}
for any $({\mathbf{x}}^*,{\bm{\lambda}}^*) \in \mathcal{X}$, we have
\begin{align}
& \mid  F(\mathbf{x}(t_{K+1})) - F(\mathbf{x}^*) \mid  \leq \epsilon (\Delta_1 + \Delta_2 \parallel \bm{\lambda}^* \parallel),  \label{22+1} \\
& \parallel \mathbf{A}\mathbf{x}(t_{K+1}) \parallel \leq \epsilon \Delta_2,\label{22+2}
\end{align}
where $\Delta_1$ and $\Delta_2$ take on the forms of (\ref{c1}) and (\ref{c2}), respectively.
Moreover, from a start position to $\mathbf{x}(t_{K+1})$, the convergence rate is given by
\begin{align}
& \mid  F(\mathbf{x}(t_{K+1})) - F(\mathbf{x}^*) \mid \leq \mathcal{O}(\frac{1}{K}), \label{23} \\ & \parallel \mathbf{A}\mathbf{x}(t_{K+1}) \parallel \leq \mathcal{O}(\frac{1}{K}). \label{24}
\end{align}
\end{Theorem}

\begin{proof}
See Appendix \ref{tp2}.
\end{proof}

\begin{Remark}
To determine $Q_m$ by (\ref{33}) and (\ref{35}), one can choose a uniform $Q_0 = ... = Q_m = ... \geq L^{\mathrm{g}}$, $m \in \mathbb{N}$, such that (\ref{33}) and (\ref{35}) hold at all times and $ \beta \in (0, \frac{ \mu}{2H(H+D) \Pi \parallel \mathbf{A} \parallel^2}]$. Alternatively, a varying $Q_m$ means that, in slot $m$, one can choose $Q_m \in [L^{\mathrm{g}} , Q_{m-1} + \frac{\mu}{H})$, which is non-empty if $Q_{m-1} \geq L^{\mathrm{g}}$. That means, given that $Q_0 \geq L^{\mathrm{g}}$, $Q_m$ can be determined by (\ref{33}) and (\ref{35}) throughout the whole process. In the trivial case that $\parallel \mathbf{A} \parallel = 0$, as seen from Algorithm \ref{a1x}, $\beta$ can be chosen in $\mathbb{R}$.

\end{Remark}

\subsection{Distributed Realization of Asyn-PPG Algorithm}\label{rm2}

In some large-scale distributed networks, directly implementing Algorithm \ref{a1x} can be restrictive in the sense that each agent needs to access the full state information, which can be unrealizable if the communication networks are not fully connected \cite{pang2019randomized}. To overcome this issue, a promising solution is to establish a central server responsible for collecting, storing and distributing the necessary information of the system (as discussed in \cite{zhou2018distributed,ho2013more,li2014scaling}), which can also effectively relieve the storage burden of the historical data for the agents. In such a system, each agent pushes its state information, e.g., $\mathbf{x}_i(t)$, into the server and pulls the historical information, e.g., $\mathbf{x}^{\mathrm{d}}(t_m)$, from the server due to the delays between the agent side and the server.

As another distributed realization, we consider a composite objective function ${F}(\mathbf{x})=\sum_{i \in \mathcal{V}} {F}_i(\mathbf{x}_i)$ without any coupling constraint, where the agents aim to achieve an agreement on the optimal solution to $\min_{\mathbf{x}} {F}(\mathbf{x})$ by optimizing private functions ${F}_i(\mathbf{x}_i)$, $\forall i\in \mathcal{V}$. To this end, we can apply graph theory and consensus protocol $\mathbf{x}_1 = ... = \mathbf{x}_{N}$ if $\mathcal{G}$ is connected. In this case, $\mathbf{A}$ can be designed as ($\mathbf{A}_{ei}$ is the $(e,i)$th sub-block of $\mathbf{A}$)
\begin{align}\label{}
\mathbf{A}_{e \in {\mathcal{E}},i \in \mathcal{V}} = \left\{\begin{array}{cc}
                 \mathbf{I}_M & \hbox{if }e = (i,j) \hbox{ $\&$ } i < j, \\
                 -\mathbf{I}_M & \hbox{if }e = (i,j) \hbox{ $\&$ } i  > j, \\
                 \mathbf{O}_{M \times M} & \hbox{otherwise,}
               \end{array}
               \right.
\end{align}
which is an augmented incidence matrix of $\mathcal{G}$ \cite{dimarogonas2010stability}. It can be checked that the consensus of $\mathbf{x}$ can be defined by $\mathbf{A}\mathbf{x}=\mathbf{0}$.  Then, we can have $\mathbf{W} = \mathbf{A}^{\top} \mathbf{A} = \mathbf{L} \otimes \mathbf{I}_M \in \mathbb{R}^{NM \times NM }$ with $\mathbf{W}_{ij} \neq \mathbf{O}_{M \times M}$ only if $(i,j) \in \mathcal{E}$ or $i = j$ \cite{sun2017distributed}. Hence, the updating of $\mathbf{x}_i$ in $\mathcal{T}_i$ can be written as
\begin{align}\label{}
         \mathbf{x}_i(t^{(n)}_{m}+1) = & \mathrm{prox}_{h_i}^{\eta_i(t^{(n)}_m)} (\mathbf{x}_i(t^{(n)}_m) - \eta_i(t^{(n)}_m) (\nabla f_i(\mathbf{x}_i(t^{(n)}_m)) \nonumber \\
        + & \frac{\beta }{\alpha_i (t_{m+1}-1)} \sum_{j\in \mathcal{V}_i \cup \{i\} }\mathbf{W}_{ij}\mathbf{x}_j^{\mathrm{d}}(t_m))),
\end{align}
which means the agent only needs to access the delayed information from neighbours. See an example in simulation A.

\section{Numerical Simulation}\label{se6}

In this section, we discuss two motivating applications of the proposed Asyn-PPG algorithm.
\subsection{Consensus Based Distributed LASSO Problem}\label{ss1}

\begin{figure}[htbp]
  \centering
  \includegraphics[height=3.5cm,width=4.5cm]{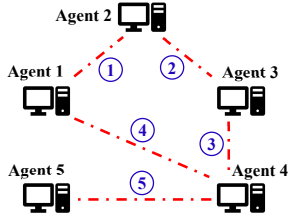}\\
  \caption{Communication typology of the 5-agent SAN.}\label{gr}
\end{figure}

In this subsection, the feasibility of the Asyn-PPG algorithm will be demonstrated by solving a consensus based distributed LASSO problem in a connected and undirected 5-agent SAN ${\mathcal{G}} = \{{\mathcal{V}}, {\mathcal{E}} \}$. The communication topology is designed in Fig. \ref{gr}.

In this problem, the global cost function is considered as $\widetilde{F}_A(\tilde{\mathbf{x}})= \frac{1}{2} \sum_{i \in \mathcal{V}} \parallel \mathbf{P}_i\tilde{\mathbf{x}} - \mathbf{q}_i \parallel^2 + \varrho \parallel \tilde{\mathbf{x}} \parallel_1$, $\tilde{\mathbf{x}} \in \mathbb{R}^5$, $\varrho >0$. To realize a consensus based distributed computation fashion, inspired by \cite{bazerque2010distributed}, the local cost function of agent $i$ is designed as $F_{A,i} (\mathbf{x}_i) = \frac{1}{2} \parallel \mathbf{P}_i\mathbf{x}_i - \mathbf{q}_i \parallel^2 + \frac{\varrho}{|\mathcal{V}|} \parallel \mathbf{x}_i \parallel_1$, $\mathbf{x}_i \in \mathbb{R}^5$. The idea of generating the data follows the method introduced in \cite{parikh2014proximal}. Firstly, we generate a $(5 \times 5)$-dimensional matrix $\mathbf{P}'_i$, where each element is generated by a normal distribution $\mathcal{N}(0,1)$. Then, normalize the columns of $\mathbf{P}'_i$ to have $\mathbf{P}_i \in \mathbb{R}^{5 \times 5}$. $\mathbf{q}_i$ is generated by $\mathbf{q}_i=\mathbf{P}_i\hat{\mathbf{x}}_i+\bm{\delta}_i$, where $\hat{\mathbf{x}}_i \in \mathbb{R}^{5}$ is certain given vector and $\bm{\delta}_i \sim \mathcal{N}(\mathbf{0},10^{-3}\mathbf{I}_5)$ is an additive noise, $\forall i \in {\mathcal{V}}$.

Then, the consensus based distributed LASSO problem can be formulated as the following linearly constrained optimization problem:
\begin{align}\label{sta}
\hbox{\textbf{(P2)}}: \quad \min_{\mathbf{x}} \quad &  F_A(\mathbf{x})=\sum_{i \in \mathcal{V}} F_{A,i} (\mathbf{x}_i) \nonumber \\
\hbox{subject to} \quad & \mathbf{M} \mathbf{x}=\mathbf{0},
\end{align}
where $ \mathbf{M}$ is generated by the method introduced in section \ref{rm2}, i.e.,
\begin{align}\label{}
 \mathbf{M} = \left(
            \begin{array}{ccccc}
              \mathbf{I}_5 & -\mathbf{I}_5 & \mathbf{O}_{5 \times 5}  & \mathbf{O}_{5 \times 5}  &   \mathbf{O}_{5 \times 5} \\
               \mathbf{O}_{5 \times 5} & \mathbf{I}_5 & -\mathbf{I}_5 & \mathbf{O}_{5 \times 5}  &  \mathbf{O}_{5 \times 5} \\
               \mathbf{O}_{5 \times 5} & \mathbf{O}_{5 \times 5}  & \mathbf{I}_5 & -\mathbf{I}_5 & \mathbf{O}_{5 \times 5}  \\
              \mathbf{I}_5 &  \mathbf{O}_{5 \times 5} & \mathbf{O}_{5 \times 5}  & -\mathbf{I}_5 & \mathbf{O}_{5 \times 5}  \\
              \mathbf{O}_{5 \times 5}  & \mathbf{O}_{5 \times 5}  & \mathbf{O}_{5 \times 5}  & \mathbf{I}_5 & -\mathbf{I}_5 \\
            \end{array}
          \right) ,
\end{align}
and $\mathbf{x}=(\mathbf{x}^{\top}_1,\mathbf{x}^{\top}_2,...,\mathbf{x}^{\top}_{5})^{\top}$. 

\subsubsection{{Simulation Setup}}
The width of time slots is set as $H=10$ and the upper bound of communication delays is set as $D=2$. To represent the ``worst delays'', we let $\tau (t_m) = t_m- D$, $\forall m \in \mathbb{N}_+$. In slot $m$, the update frequency of agent $i$ is chosen from $P_{i,m} \in \{1,2,...,H\}$, and the action instants are randomly determined. $\varrho$ is set as $10$. Other settings for $\alpha_i$, $\eta_i$ and $\beta$ are consistent with the conditions specified in {Theorem 2}, $i \in \mathcal{V}$. To show the dynamics of the convergence error, we let $\mathbf{x}^*$ be the optimal solution to Problem (P2) and define $\gamma_A(t)=  | F_A(\mathbf{x}(t)) - F_A(\mathbf{x}^*) |$, $t \in \mathcal{T}$.

\subsubsection{{Simulation Result}}
By Algorithm \ref{a1x}, the simulation result is shown in Figs. \ref{g2}-(a) to \ref{g2}-(c).
The action clock of the agents is depicted in Fig. \ref{g2}-(a). 
By performing Algorithm \ref{a1x}, Fig. \ref{g2}-(b) shows the dynamics of decision variables of all the agents. It can be noted that all the trajectories of the agents converge to $\mathbf{x}^*$. The dynamics of $\gamma_A(t)$ is shown in Fig. \ref{g2}-(c).

\begin{figure*}
\centering
\subfigure[Action clock of the agents. ``1'' represents ``action'' and ``0'' represents ``non-action''.]{
\begin{minipage}[t]{0.33\linewidth}
\centering
\includegraphics[width=6.5cm]{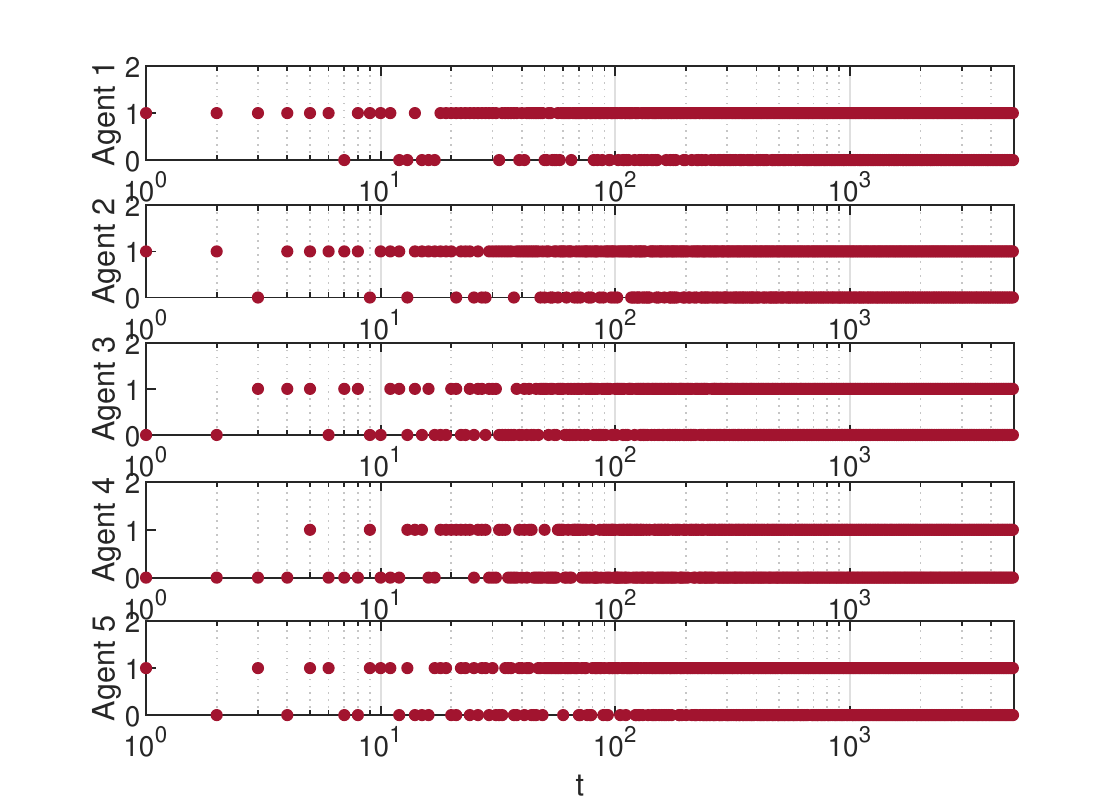}
\end{minipage}%
}%
\subfigure[Dynamics of the state of the agents.]{
\begin{minipage}[t]{0.33\linewidth}
\centering
\includegraphics[width=6.5cm]{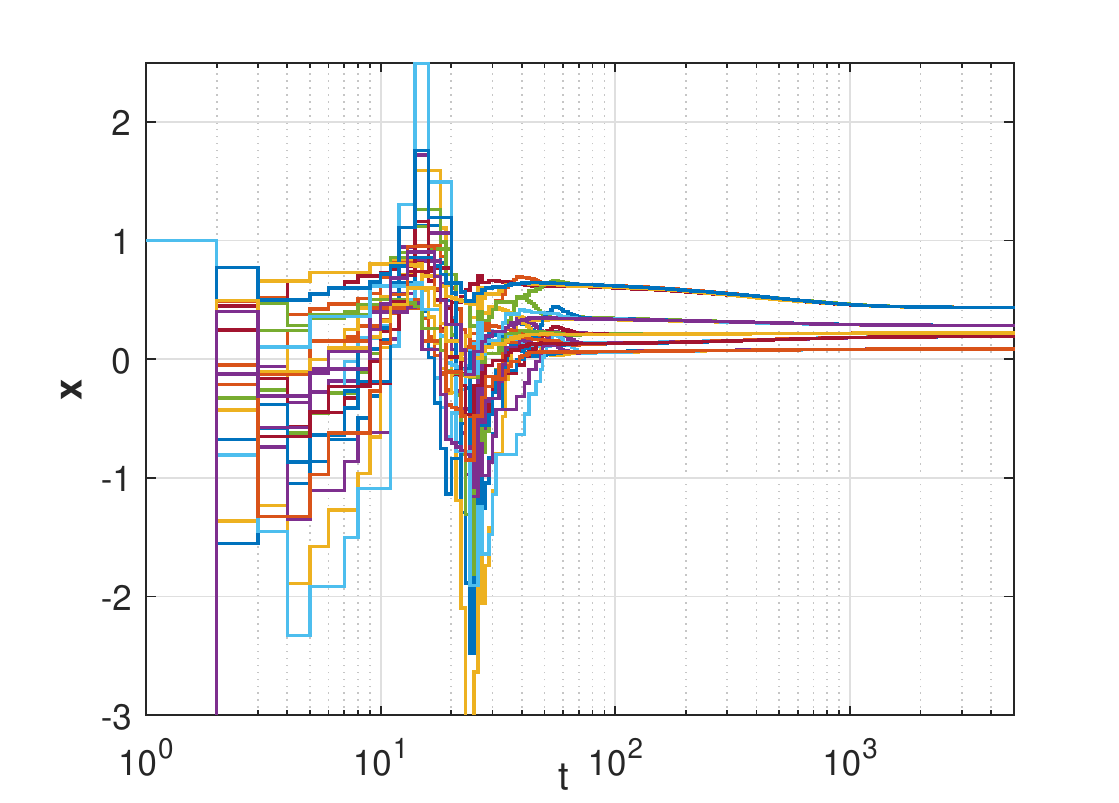}
\end{minipage}%
}%
\subfigure[Dynamics of convergence error $\gamma_A(t)$.]{
\begin{minipage}[t]{0.33\linewidth}
\centering
\includegraphics[width=6.5cm]{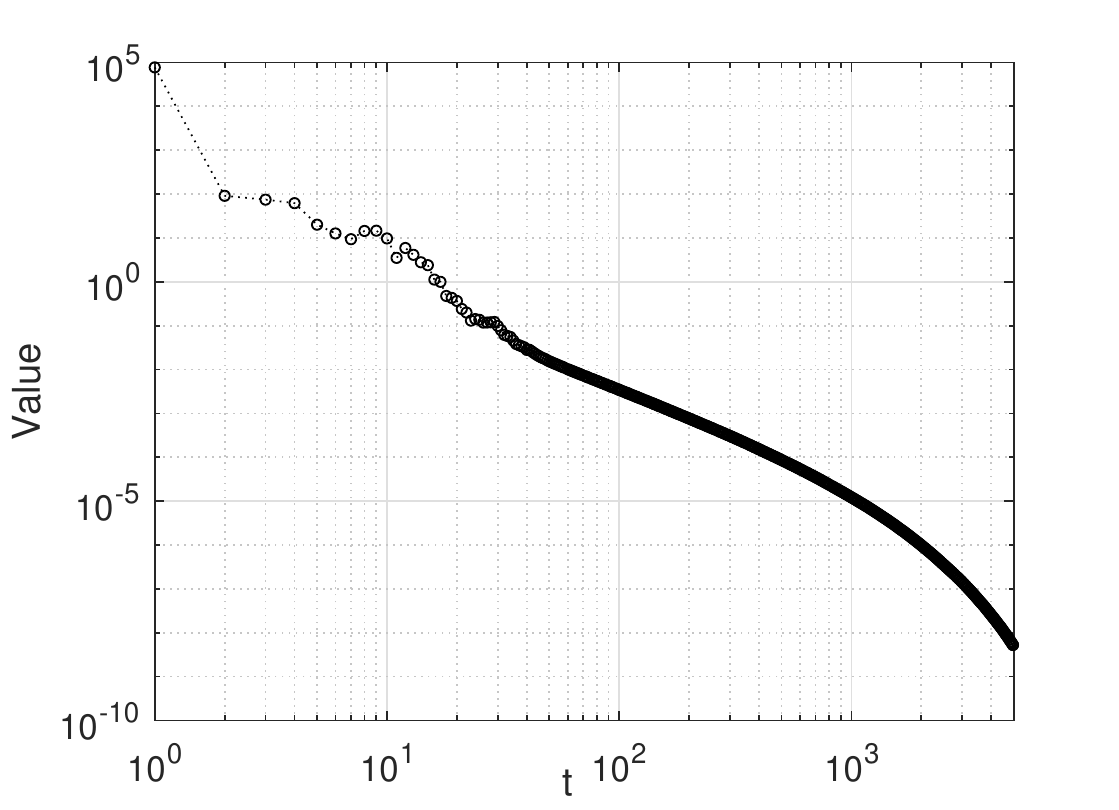}
\end{minipage}%
}%
\caption{Result of Simulation A with $H=10$ and $D=2$.}\label{g2}
\end{figure*}

\begin{figure*}
\centering
\subfigure[Action clock of UCs and users. ``1'' represents ``action'' and ``0'' represents ``non-action''.]{
\begin{minipage}[t]{0.33\linewidth}
\centering
\includegraphics[width=6.5cm]{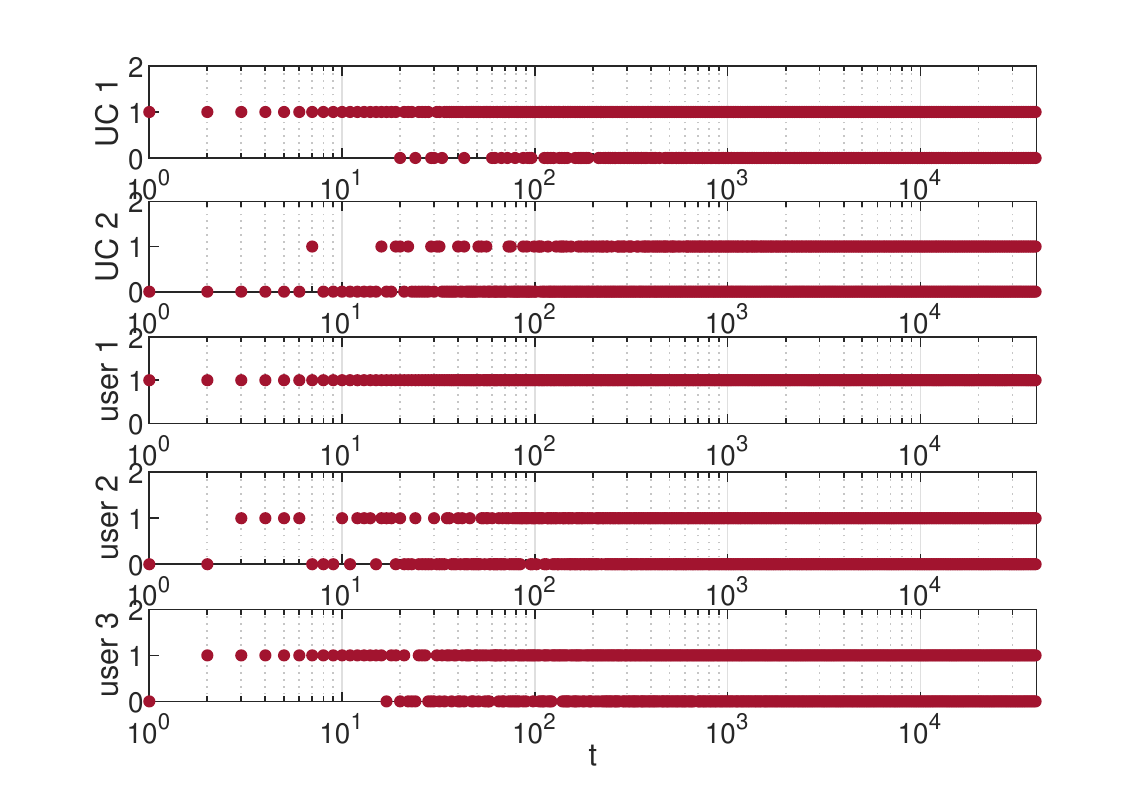}
\end{minipage}%
}%
\subfigure[Dynamics of the state of UCs and users.]{
\begin{minipage}[t]{0.33\linewidth}
\centering
\includegraphics[width=6.5cm]{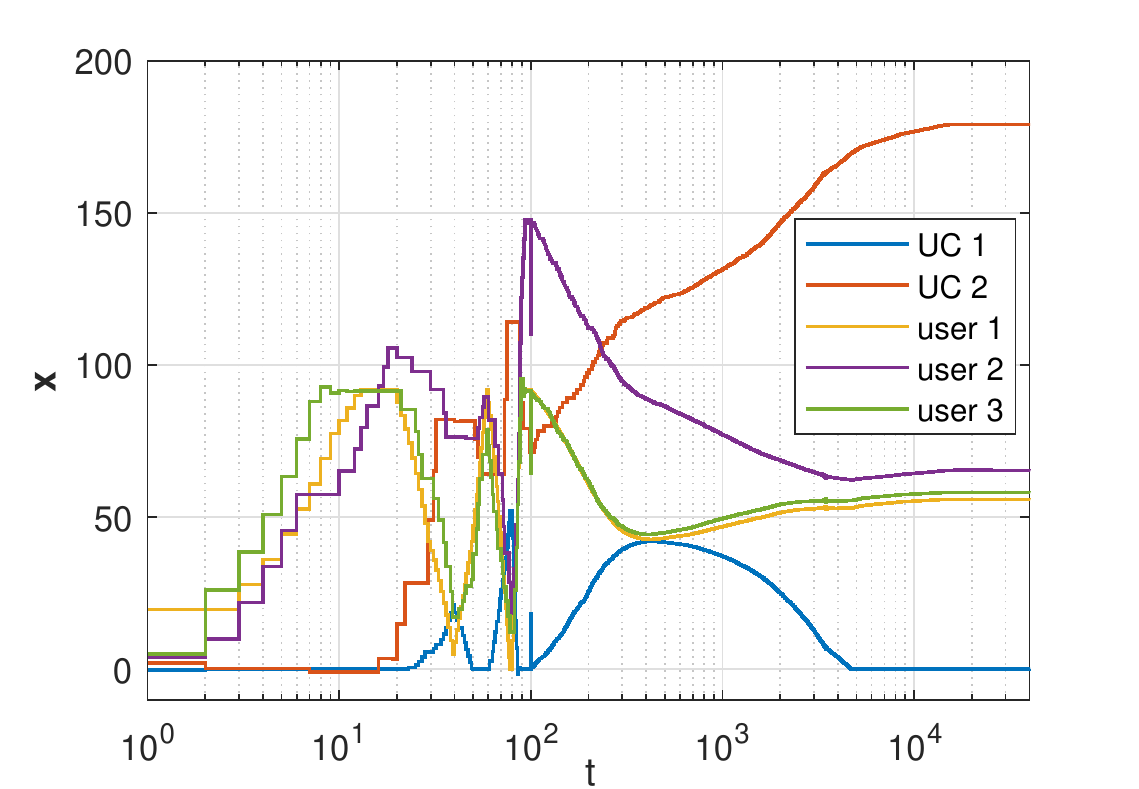}
\end{minipage}%
}%
\subfigure[Dynamics of convergence error $\gamma_B(t)$.]{
\begin{minipage}[t]{0.33\linewidth}
\centering
\includegraphics[width=6.5cm]{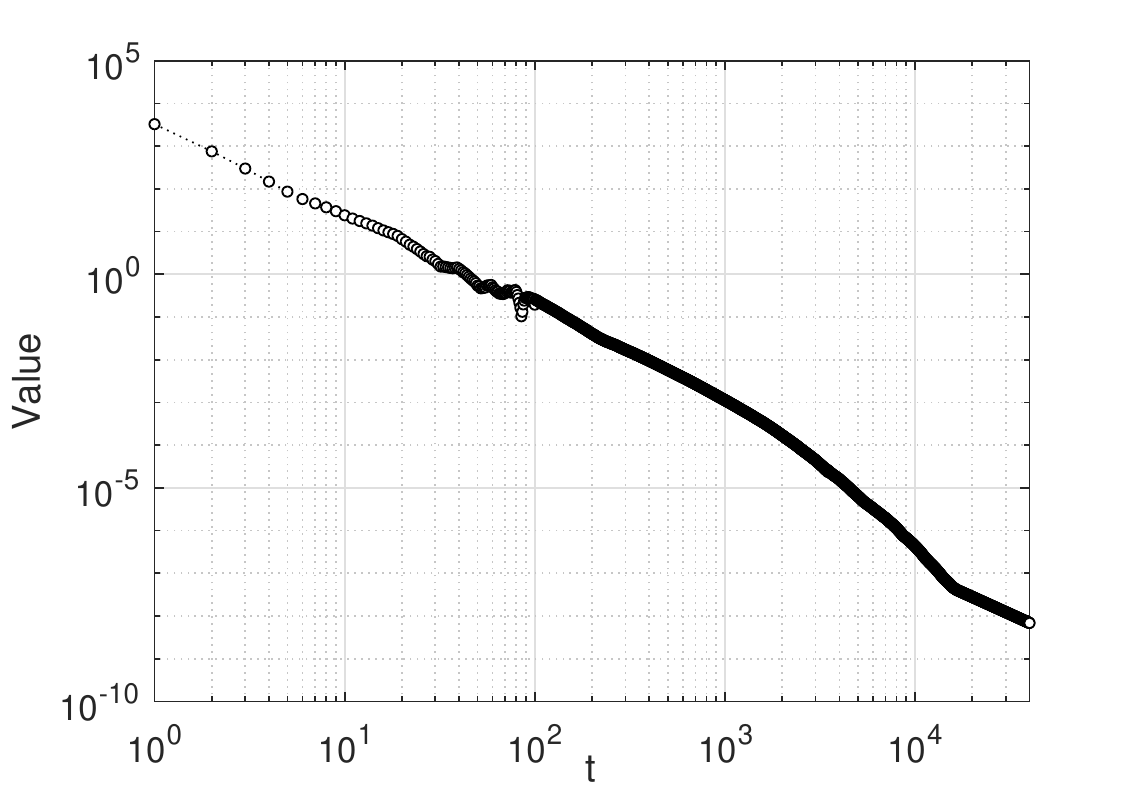}
\end{minipage}%
}%
\caption{Result of Simulation B with $H=15$ and $D=5$.}\label{g3}
\end{figure*}

\subsection{Social Welfare Optimization Problem in Electricity Market}


In this subsection, we verify the feasibility of our proposed Asyn-PPG algorithm by solving a social welfare optimization problem in the electricity market with 2 utility companies (UCs) and 3 users.

The social welfare optimization problem is formulated as
\begin{align}\label{}
 \textbf{(P3)}: \quad \min_{\mathbf{x}} \quad  &  \sum_{i \in \mathcal{V}_{{\textrm{UC}}}} C_i(x^{{\textrm{UC}}}_{i}) - \sum_{j \in \mathcal{V}_{\textrm{user}}} U_j(x^{\textrm{user}}_{j}) \nonumber\\
  \hbox{subject to} \quad & \sum_{i \in \mathcal{V}_{{\textrm{UC}}}} x^{{\textrm{UC}}}_{i} = \sum_{j \in \mathcal{V}_{\textrm{user}}} x^{\textrm{user}}_{j}, \label{s1} \\
  & x^{{\textrm{UC}}}_{i} \in [0,x^{{\textrm{UC}}}_{i,\textrm{max}}], \quad \forall i \in  \mathcal{V}_{{\textrm{UC}}} \\
  & x^{\textrm{user}}_{j} \in [0,x^{\textrm{user}}_{j,\textrm{max}}], \quad \forall j \in  \mathcal{V}_{\textrm{user}}
\end{align}
where $\mathcal{V}_{{\textrm{UC}}}$ and $\mathcal{V}_{\textrm{user}}$ are the sets of UCs and users, respectively. $\mathbf{x} = (x^{{\textrm{UC}}}_{1},...,x^{{\textrm{UC}}}_{|\mathcal{V}_{{\textrm{UC}}}|}, x^{\textrm{user}}_{1},...,x^{\textrm{user}}_{|\mathcal{V}_{\textrm{user}}|})^{\top}$ with $x^{{\textrm{UC}}}_{i}$ and $x^{\textrm{user}}_{j}$ being the quantities of energy generation and consumption of UC $i$ and user $j$, respectively. $C_i(x^{{\textrm{UC}}}_{i})$ is the cost function of UC $i$ and $U_j(x^{\textrm{user}}_{j})$ is the utility function of user $j$, $i \in \mathcal{V}_{{\textrm{UC}}}$, $j \in \mathcal{V}_{\textrm{user}}$. Constraint (\ref{s1}) ensures the supply-demand balance in the market. $x^{{\textrm{UC}}}_{i,\textrm{max}}>0$ and $x^{\textrm{user}}_{j,\textrm{max}}>0$ are the upper bounds of $x^{{\textrm{UC}}}_{i}$ and $x^{\textrm{user}}_{j}$, respectively.
The detailed expressions of $C_i(x^{{\textrm{UC}}}_{i})$ and $U_j(x^{\textrm{user}}_{j})$ are designed as \cite{pourbabak2017novel}
\begin{align}\label{}
& C_i(x^{{\textrm{UC}}}_{i}) = \kappa_i (x^{{\textrm{UC}}}_{i})^2 + \xi_i x^{{\textrm{UC}}}_{i} +\varpi_i, \nonumber  \\
& U_j(x^{\textrm{user}}_{j}) =
\left\{\begin{array}{ll}
\nu_j x^{\textrm{user}}_{j} - \varsigma_j (x^{\textrm{user}}_{j})^2, & x^{\textrm{user}}_{j} \leq \frac{\nu_j}{2\varsigma_j} \\
  \frac{\nu_j^2}{4 \varsigma_j}, & x^{\textrm{user}}_{j} > \frac{\nu_j}{2\varsigma_j} \nonumber
\end{array}
\right.
\end{align}
respectively, where $\kappa_i,\xi_i,\varpi_i,\nu_j,\varsigma_j$ are all parameters, $\forall i \in  \mathcal{V}_{{\textrm{UC}}}$, $\forall j \in  \mathcal{V}_{\textrm{user}}$.

Note that the structure of Problem (P3) can be modified into that of (P1) with the method introduced in Remark \ref{re12}.
By some direct calculations, the optimal solution to Problem (P3) can be obtained as $\mathbf{x}^* = (0, 179.1, 55.51, 65.84, 57.75)^{\top}$. Define $F_B(\mathbf{x})=  \sum_{i \in \mathcal{V}_{{\textrm{UC}}}} C_i(x^{{\textrm{UC}}}_{i}) - \sum_{j \in \mathcal{V}_{\textrm{user}}} U_j(x^{\textrm{user}}_{j})$ and $\gamma_B(t) = |F_B(\mathbf{x}(t))-F_B(\mathbf{x}^*)|$, $t \in \mathcal{T}$.

\subsubsection{{Simulation Setup}}
The parameters of this simulation are listed in Table I \cite{pourbabak2017novel}. The width of slots and the upper bound of communication delays are set as $H=15$ and $D=5$, respectively. In addition, to test the performance of the Asyn-PPG algorithm with large heterogeneity of the update frequencies, the percentages of action instants of UC 1, UC 2, user 1, user 2, and user 3 are set around $80\%$, $20\%$, $100\%$, $50\%$, and $70\%$, respectively.

\subsubsection{{Simulation Result}}

The simulation result is shown in Figs. \ref{g3}-(a) to \ref{g3}-(c). Fig. \ref{g3}-(a) shows the action clock of UCs and users. Fig. \ref{g3}-(b) shows the dynamics of the decision variables of them. The dynamics of convergence error is shown in Fig. \ref{g3}-(c). It can be seen that their states converge to the optimal solution $\mathbf{x}^*$. Notably, due to the local constraints on the variables, the optimal supply quantities of UC 1 and UC 2 reach the lower and upper bounds, respectively, and other variables converge to interior optimal positions.

\begin{table}\label{t2}
\caption{Parameters of UCs and users}
\label{tab2}
\begin{center}
\begin{tabular}{cccccccc}
\bottomrule
& \multicolumn{4}{l}{$ \quad \quad \quad \quad \quad$ \quad UCs} & \multicolumn{3}{l}{$\quad \quad \quad$ Users} \\
\hline
$i/j$ & $\kappa_i$ & $\xi_i$ & $\varpi_i$ & $x^{\textrm{UC}}_{i,\textrm{max}}$ & $\nu_j$ &  $\varsigma_j$  & $x^{\textrm{user}}_{j,\textrm{max}}$ \\
\hline
1 & 0.0031& 8.71 & 0 & 113.23  & 17.17 & 0.0935 & 91.79 \\
2 & 0.0074 & 3.53 & 0 & 179.1 & 12.28 & 0.0417 & 147.29 \\
3 & - & - & - & - & 18.42 & 0.1007 & 91.41 \\
\bottomrule
\end{tabular}
\end{center}
\end{table}

\section{Conclusion}\label{se7}

In this work, we proposed an Asyn-PPG algorithm for solving a linearly constrained composite optimization problem in a multi-agent network. An SAN model was established where the agents are allowed to update asynchronously with possibly outdated information of other agents. Under such a framework, a periodic convergence with rate $\mathcal{O}(\frac{1}{K})$ is achieved. As the main feature, the theoretical analysis of the Asyn-PPG algorithm is based on deterministic derivation, which is advantageous over the stochastic method which relies on the acquisition of large-scale historical data. The distributed realization of Asyn-PPG algorithm in some specific networks and problems was also discussed.


\appendix{}

\subsection{Proof of Proposition \ref{pp1}}\label{pp1p}

(\ref{e1}) can be directly proved with (\ref{pan2}).

(\ref{e6}) and (\ref{e7}) can be directly proved with (\ref{pan3}).

By Algorithm \ref{a1x}, $\mathbf{x}_i (t)$ remains unchanged if $t \in [t^{(n)}_m + 1, t_m^{(n+1)}]$. So (\ref{e0}) holds.

For (\ref{e2}), $\mathbf{x}_i(t)$ remains unchanged if $t \in [t^{(P_{i,m})}_{m} + 1, t^{(P_{i,m}+1)}_{m}]$. Since $t_{m+1 }\in [t^{(P_{i,m})}_{m} + 1, t^{(P_{i,m}+1)}_{m}]$ (see (\ref{pan2}) and (\ref{pan1})), then $\mathbf{x}_i(t_{m+1}) = \mathbf{x}_i(t^{(P_{i,m}+1)}_{m})$.

(\ref{e4}) and (\ref{e4+1}) can be jointly verified since $\alpha_i(t)$ remains unchanged during the interval $[t^{(P_{i,m})}_{m}, t^{(P_{i,m}+1)}_{m}-1]$ and $(t_{m+1}-1) \in [t^{(P_{i,m})}_{m}, t^{(P_{i,m}+1)}_{m}-1]$ (see (\ref{pan2}) and (\ref{pan1})).

The proofs of (\ref{e5}) and (\ref{e5+1}) are similar to those of (\ref{e4}) and (\ref{e4+1}) since the values of $\alpha_i(t)$ and $\eta_i(t)$ are updated simultaneously in $\mathcal{T}_i$.

\subsection{Proof of Proposition \ref{la2}}\label{la2p}
For (\ref{l1}), we have
\begin{align}\label{}
  & \parallel \mathbf{x} (t_{m+1}) -  \mathbf{x}^{\mathrm{d}}(t_{m+1}) \parallel^2
  =  \sum_{i \in \mathcal{V}} \parallel \mathbf{x}_i (t_{m+1}) -  \mathbf{x}_i^{\mathrm{d}}(t_{m+1}) \parallel^2 \nonumber \\
 & = \sum_{i \in \mathcal{V}}  \parallel \mathbf{x}_i (t_{m+1}) -  \mathbf{x}_i(t^{(n_{i,m})}_m) \parallel^2 \nonumber \\
 & \leq \sum_{i \in \mathcal{V}} (P_{i,m} - n_{i,m} +1 )\sum_{n = n_{i,m}}^{P_{i,m}} \parallel \mathbf{x}_i (t^{(n+1)}_{m}) -  \mathbf{x}_i(t^{(n)}_{m}) \parallel^2  \nonumber \\
 & \leq \sum_{i \in \mathcal{V}}  \sum_{n = 1}^{P_{i,m}} D \parallel \mathbf{x}_i (t^{(n+1)}_{m}) -  \mathbf{x}_i(t^{(n)}_{m}) \parallel^2,
\end{align}
where the first inequality holds by {\em{Cauchy-Schwarz inequality}} and (\ref{e2}), and the second inequality holds with (\ref{6}). Similarly, for (\ref{l2}),
\begin{align}\label{}
  & \parallel \mathbf{x} (t_{m+1}) -  \mathbf{x}(t_m) \parallel^2
  =  \sum_{i \in \mathcal{V}} \parallel \mathbf{x}_i (t^{(P_{i,m}+1)}_m) -  \mathbf{x}_i(t^{(1)}_m) \parallel^2 \nonumber \\
 & \leq \sum_{i \in \mathcal{V}} \sum_{n = 1}^{P_{i,m}} P_{i,m} \parallel \mathbf{x}_i (t^{(n+1)}_{m}) -  \mathbf{x}_i(t^{(n)}_{m}) \parallel^2   \nonumber \\
 & \leq \sum_{i \in \mathcal{V}} \sum_{n = 1}^{P_{i,m}} H \parallel \mathbf{x}_i (t^{(n+1)}_{m}) -  \mathbf{x}_i(t^{(n)}_{m}) \parallel^2,
\end{align}
where {\em{Cauchy-Schwarz inequality}}, (\ref{e1}), and (\ref{e2}) are used.
%

\subsection{Proof of Lemma \ref{lm3}}\label{lm3p}
By (\ref{th12+1}) and (\ref{lp1}), we have
\begin{align}\label{e10}
& \frac{1}{\alpha_i(t^{(P_{i,m})}_m)} -  \frac{1}{\alpha_i(t^{(P_{i,m-1})}_{m-1})} =
\frac{1}{\alpha_i(t^{(P_{i,m})}_m)} - \frac{1}{\alpha_i(t^{(0)}_{m})} \nonumber \\
& = \sum_{n=1}^{P_{i,m}} ( \frac{1}{\alpha_i(t^{(n)}_m)} - \frac{1}{\alpha_i(t^{(n-1)}_{m})} ) = \sum_{n=1}^{P_{i,m}} \frac{\theta_i(t^{(n)}_m)}{\alpha_i(t^{(n)}_m)} = 1,
\end{align}
which is a constant, $\forall i \in \mathcal{V}$. The first equality holds due to (\ref{e6}). (\ref{e10}) means that $ \{ \frac{1}{\alpha_i(t^{(P_{i,m})}_{m})} \}_{m \in \mathbb{N}}$ is an arithmetic sequence with start element $\frac{1}{\alpha_i(t^{(P_{i,0})}_0 )}$ and common difference 1. Given that ${\alpha_1(t^{(P_{1,0})}_{0})}= ... = {\alpha_{N}(t^{(P_{{N},0})}_{0})}$, we can have $\alpha_1(t_m^{(P_{1,m})}) = ...=\alpha_{N}(t_m^{(P_{{N},m})})$, $m \in \mathbb{N}$. Then, by {Proposition \ref{pp1}}, (\ref{tha1}) can be verified.

In addition, by {Proposition \ref{pp1}}, (\ref{cc}) and the arithmetic sequence $\{ \frac{1}{\alpha_i (t^{(P_{i,m})}_{m})} \}_{m \in \mathbb{N}}$, we have
\begin{align}\label{17-4}
& \frac{1}{\alpha(t_{m}-1)}  = \frac{1}{\alpha_i(t_{m}-1)} = \frac{1}{\alpha_i(t^{(P_{i,m-1})}_{m-1})}  \nonumber \\
& = \frac{1}{\alpha_i(t^{(P_{i,0})}_{0})} + m-1  = \frac{1}{\alpha_i(t_{1}-1)} + m-1  \nonumber \\
& = \frac{1}{\alpha(t_{1}-1)} + m-1 ,
\end{align}
$\forall i \in \mathcal{V}$, which verifies (\ref{17-1}).

By (\ref{th12+1}) and (\ref{lp1}), $\{\frac{1}{\alpha_i(t^{(n)}_m)} \}_{n = 0,1,...,P_{i,m}}$ is an arithmetic sequence with start element $\frac{1}{\alpha_i(t^{(0)}_m)}$ and common difference $\frac{1}{P_{i,m}}$. Then, by {Proposition \ref{pp1}}, we have
\begin{align}\label{17-3}
\frac{1}{\alpha_i(t^{(n)}_m)} = \frac{1}{\alpha_i(t^{(0)}_m)} + \frac{n}{P_{i,m}} = \frac{1}{\alpha_i(t_m -1)} + \frac{n}{P_{i,m}}.
\end{align}
Combining (\ref{17-4}) and (\ref{17-3}) gives (\ref{17-2}).

In (\ref{17-5}), the lower bound is from the inequality chain of instants $t_m^{(n)} \leq t_{m+1} - 1 < t_{m+2} - 1$ and strictly decreasing property of $\alpha_i$ (see {Proposition \ref{pr1}}). For the upper bound, by (\ref{17-4}) and (\ref{17-3}), we have
\begin{align}\label{}
 & \frac{\alpha_i(t^{(n)}_m)}{\alpha(t_{m+2}-1)}  = \frac{(m+1)P_{i,m}\alpha(t_1-1) + P_{i,m}}{n\alpha(t_1-1) + P_{i,m} + (m-1)P_{i,m}\alpha(t_1-1)} \nonumber \\
& \leq \frac{2P_{i,m}\alpha(t_1-1) + P_{i,m}}{n\alpha(t_1-1) + P_{i,m}  }  = \frac{2\alpha(t_1-1) + 1}{\frac{n}{P_{i,m}} \alpha(t_1-1) + 1 }  \nonumber \\
& \leq \frac{2\alpha(t_1-1) + 1}{\frac{1}{H} \alpha(t_1-1) + 1 } = \Pi .
\end{align}
This verifies (\ref{17-5}).

\subsection{Proof of Lemma \ref{lpl}}\label{lplp}

By (\ref{lp1}) and {Propositions \ref{pp1}}, {\ref{pp2}}, we can have
\begin{align}\label{}
\frac{\theta_i(t_m - 1)}{\alpha_i(t_m - 1)} = \frac{\theta_i(t^{(P_{i,m-1})}_{m-1})}{\alpha_i(t^{(P_{i,m-1})}_{m-1})} = \frac{1}{P_{i,m-1}}.
\end{align}
Then, by (\ref{llp}), we can have
\begin{align}\label{}
& \frac{\theta_i(t_{m}-1)} {\alpha_i(t_{m}-1)\eta_i(t_{m}-1)} - \frac{\theta_j(t_{m}-1)} {\alpha_j(t_{m}-1)\eta_j(t_{m}-1)} \nonumber  \\
& =  \frac{1} { P_{i,m-1} \eta_i(t_{m}-1)} - \frac{1} { P_{j,m-1} \eta_i(t_{m}-1) \frac{P_{i,m-1}}{P_{j,m-1}}} \nonumber  \\
& =  0,
\end{align}
$\forall i,j \in \mathcal{V}$, which verifies (\ref{cx}).

\subsection{Proof of Lemma \ref{z}}\label{z1}
Consider instant $t_m^{(n)} \in \mathcal{T}_i \cap [t_{m},t_{m+1})$. By the proximal mapping in Algorithm \ref{a1x}, definition in (\ref{d1}), and (\ref{e0}), we have
\begin{align}\label{}
& \mathbf{x}_i(t^{(n+1)}_m)  = \mathbf{x}_i(t^{(n)}_m+1)  \nonumber \\
&  = \arg \min_{\mathbf{v} \in \mathbb{R}^M} ( h_i(\mathbf{v} )  + \frac{1}{2\eta_i(t^{(n)}_m)} \| \mathbf{v}  -  \mathbf{x}_i(t^{(n)}_m)  + \eta_i(t^{(n)}_m) \nonumber \\
& \quad \cdot (\nabla f_i(\mathbf{x}_i(t^{(n)}_m)) + \frac{\beta \mathbf{W}_i}{\alpha_i (t_{m+1}-1)}\mathbf{x}^{\mathrm{d}}(t_m)) {\|}^2 ),
\end{align}
which means
\begin{align}\label{48}
\mathbf{0} & \in  \partial h_i(\mathbf{x}_i(t^{(n+1)}_m)) + \frac{1}{\eta_i(t^{(n)}_m)} ( \mathbf{x}_i(t^{(n+1)}_m)- \mathbf{x}_i(t^{(n)}_m)  \nonumber \\
 + & \eta_i(t^{(n)}_m)
(\nabla f_i(\mathbf{x}_i(t^{(n)}_m)) +\frac{\beta \mathbf{W}_i}{\alpha_i (t_{m+1}-1)}\mathbf{x}^{\mathrm{d}}(t_m) ) ).
\end{align}
By (\ref{48}) and the convexity of $h_i$, $\forall \mathbf{z}_i \in \mathbb{R}^M$, we have
\begin{align}\label{p1}
& h_i(\mathbf{x}_i(t^{(n+1)}_m))  -h_i(\mathbf{z}_i ) \leq  \langle \nabla f_i(\mathbf{x}_i(t^{(n)}_m))  \nonumber \\
&  + \frac{\beta \mathbf{W}_i}{\alpha_i(t_{m+1}-1)}\mathbf{x}^{\mathrm{d}}(t_m), \mathbf{z}_i  - \mathbf{x}_{i}(t^{(n+1)}_m) \rangle \nonumber \\
& + \frac{1}{\eta_i(t^{(n)}_m)} \langle \mathbf{x}_i(t^{(n+1)}_m) - \mathbf{x}_i(t^{(n)}_m), \mathbf{z}_i  - \mathbf{x}_i(t^{(n+1)}_m) \rangle.
\end{align}
On the other hand, by the $L_i$-Lipschitz continuous differentiability and $\mu_i$-strong convexity of $f_i$, we have
\begin{align}\label{p2}
  f_i  (\mathbf{x}_i & (t^{(n+1)}_m))  \leq  f_i(\mathbf{x}_i (t^{(n)}_m)) + \langle \nabla f_i(\mathbf{x}_i(t^{(n)}_m)), \mathbf{x}_i(t^{(n+1)}_m)  \nonumber \\
 & - \mathbf{x}_i(t^{(n)}_m)   \rangle
  +  \frac{L_i}{2} \parallel \mathbf{x}_i(t^{(n+1)}_m)-\mathbf{x}_i(t^{(n)}_m) \parallel^2 \nonumber \\
  & \leq    f_i(\mathbf{z}_i) + \langle \nabla f_i(\mathbf{x}_i(t^{(n)}_m)), \mathbf{x}_i(t^{(n+1)}_m)- \mathbf{z}_i   \rangle   \nonumber \\
  & +  \frac{L_i}{2} \parallel \mathbf{x}_i(t^{(n+1)}_m)-\mathbf{x}_i(t^{(n)}_m) \parallel^2  \nonumber \\
 & -\frac{\mu_i}{2}\parallel \mathbf{z}_i-\mathbf{x}_i(t^{(n)}_m) \parallel^2.
\end{align}
Adding (\ref{p1}) and (\ref{p2}) together from the both sides gives
\begin{align}\label{p2+1}
& F_i (\mathbf{x}_i  (t^{(n+1)}_m))  -  {F}_i(\mathbf{z}_i)
\leq  \langle \mathbf{A}_i^{\top} \bm{\lambda}_i^{\mathrm{d}}(t_m), \mathbf{z}_i - \mathbf{x}_{i}(t^{(n+1)}_m)\rangle \nonumber \\
 & +  \frac{1}{\eta_i(t^{(n)}_m)}  \langle \mathbf{x}_i(t^{(n+1)}_m) - \mathbf{x}_i(t^{(n)}_m), \mathbf{z}_i - \mathbf{x}_i(t^{(n+1)}_m)\rangle \nonumber \\
  & +  \frac{L_i}{2}  \parallel \mathbf{x}_i (t^{(n+1)}_m)-\mathbf{x}_i (t^{(n)}_m) \parallel^2   -\frac{\mu_i}{2}\parallel \mathbf{z}_i  - \mathbf{x}_i (t^{(n)}_m) \parallel^2,
\end{align}
where
\begin{equation}\label{}
  \bm{\lambda}_i^{\mathrm{d}}(t_m)= \frac{\beta \mathbf{A}\mathbf{x}^{\mathrm{d}}(t_m)}{\alpha_i(t_{m+1}-1)}.
\end{equation}
By letting $\mathbf{z}_i = \mathbf{x}_i^*$ and $\mathbf{z}_i = \mathbf{x}_i(t_m^{(n)})$ in (\ref{p2+1}), we have
\begin{align}\label{p4}
& {F}_i(\mathbf{x}_i (t_m^{(n+1)}))  - {F}_i(\mathbf{x}_i^*)
\leq  \langle \mathbf{A}_i^{\top} \bm{\lambda}_i^{\mathrm{d}}(t_m), \mathbf{x}_i^* - \mathbf{x}_i(t_m^{(n+1)})\rangle \nonumber \\
& + \frac{1}{\eta_i(t_m^{(n)})} \langle \mathbf{x}_i(t_m^{(n+1)}) - \mathbf{x}_i(t_m^{(n)}), \mathbf{x}_i^* - \mathbf{x}_i(t_m^{(n)})\rangle  \nonumber \\
&  + \frac{1}{2}(L_i-\frac{2}{\eta_i(t_m^{(n)})})  \parallel \mathbf{x}_i(t_m^{(n+1)})-\mathbf{x}_i(t_m^{(n)}) \parallel^2  \nonumber \\
& -\frac{\mu_i}{2}\parallel \mathbf{x}_i^*-\mathbf{x}_i(t_m^{(n)}) \parallel^2 ,
\end{align}
and
\begin{align}\label{p5}
 & {F}_i  (\mathbf{x}_i(t_m^{(n+1)}))  - {F}_i(\mathbf{x}_i(t_m^{(n)}))
\leq  \langle \mathbf{A}_i^{\top} \bm{\lambda}_i^{\mathrm{d}}(t_m), \mathbf{x}_i(t_m^{(n)}) \nonumber \\
 &  \quad \quad - \mathbf{x}_i(t_m^{(n+1)}) \rangle + \frac{1}{2}({L_i}- \frac{2}{\eta_i(t_m^{(n)})} ) \nonumber \\
  & \quad \quad \cdot \parallel \mathbf{x}_i(t_m^{(n+1)})-\mathbf{x}_i(t_m^{(n)}) \parallel^2.
\end{align}
Then, by multiplying (\ref{p4}) by $\theta_i(t_m^{(n)})$ and (\ref{p5}) by $1-\theta_i(t_m^{(n)})$ and adding the results together, we have
\begin{align}\label{p6}
 &{F}_i(\mathbf{x}_i(t_m^{(n+1)})) - {F}_i(\mathbf{x}_i^*)  \nonumber \\
 & - (1-\theta_i(t_m^{(n)}))({F}_i(\mathbf{x}_i(t_m^{(n)})) - {F}_i(\mathbf{x}_i^*)) \nonumber \\
 \leq & -\langle \mathbf{A}_i^{\top} \bm{\lambda}_i^{\mathrm{d}}(t_m), \mathbf{x}_i(t_m^{(n+1)})  - (1-\theta_i(t_m^{(n)}))\mathbf{x}_i(t_m^{(n)})\rangle \nonumber \\
& + \frac{1}{2}(L_i- \frac{2}{\eta_i(t_m^{(n)})})\parallel \mathbf{x}_i(t_m^{(n+1)})-\mathbf{x}_i(t_m^{(n)}) \parallel^2 \nonumber \\
& + \theta_i(t_m^{(n)})\langle \mathbf{A}_i^{\top} \bm{\lambda}_i^{\mathrm{d}}(t_m), \mathbf{x}_i^* \rangle -\frac{\theta_i(t_m^{(n)})\mu_i}{2}\parallel \mathbf{x}_i^*  -\mathbf{x}_i(t_m^{(n)}) \parallel^2 \nonumber \\
& + \frac{\theta_i(t_m^{(n)})}{\eta_i(t_m^{(n)})} \langle \mathbf{x}_i(t_m^{(n+1)}) - \mathbf{x}_i(t_m^{(n)}), \mathbf{x}_i^* - \mathbf{x}_i(t_m^{(n)})\rangle \nonumber \\
 = & - \frac{\alpha_i(t_m^{(n)}) }{\beta} \langle \bm{\lambda}_i^{\mathrm{d}}(t_m), \bm{\lambda}_i(t_m^{(n+1)}) - \bm{\lambda}_i(t_m^{(n)}) \rangle \nonumber \\
& + \frac{1}{2}({L_i}- \frac{2}{\eta_i(t_m^{(n)})})\parallel \mathbf{x}_i(t_m^{(n+1)})-\mathbf{x}_i(t_m^{(n)}) \parallel^2 \nonumber \\
& +\theta_i(t_m^{(n)}) \langle \bm{\lambda}_i^{\mathrm{d}}(t_m),  \mathbf{A}_i \mathbf{x}_i^* \rangle -\frac{\theta_i(t_m^{(n)})\mu_i}{2}\parallel \mathbf{x}_i^* -\mathbf{x}_i(t_m^{(n)}) \parallel^2 \nonumber \\
&  + \frac{\theta_i(t_m^{(n)})} {\eta_i(t_m^{(n)})} \underbrace{\langle \mathbf{x}_i(t_m^{(n+1)}) - \mathbf{x}_i(t_m^{(n)}), \mathbf{x}_i^* - \mathbf{x}_i(t_m^{(n)})\rangle}_{=  \Gamma_1} \nonumber \\
  = & - \frac{\alpha_i(t_m^{(n)})}{\beta} \langle \bm{\lambda}_i^{\mathrm{d}}(t_m), \bm{\lambda}_i(t_m^{(n+1)}) - \bm{\lambda}_i(t_m^{(n)}) \rangle  \nonumber \\
&  + \frac{1}{2}({L_i}- \frac{2-\theta_i(t_m^{(n)})}{\eta_i(t_m^{(n)})} ) \parallel \mathbf{x}_i(t_m^{(n+1)}) -\mathbf{x}_i(t_m^{(n)}) \parallel^2 \nonumber \\
& +( \frac{\theta_i(t_m^{(n)})}{2\eta_i(t_m^{(n)})}  -\frac{\theta_i(t_m^{(n)})\mu_i}{2} )\parallel \mathbf{x}_i^*-\mathbf{x}_i(t_m^{(n)}) \parallel^2 \nonumber \\
 & -   \frac{\theta_i(t_m^{(n)})}{2\eta_i(t_m^{(n)})} \parallel \mathbf{x}_i^*-\mathbf{x}_i(t_m^{(n+1)}) \parallel^2 \nonumber \\
 & + \theta_i(t_m^{(n)}) \langle \bm{\lambda}_i^{\mathrm{d}}(t_m),  \mathbf{A}_i \mathbf{x}_i^* \rangle,
\end{align}
where
\begin{equation}\label{44}
  \bm{\lambda}_i(t_m^{(n)}) = \frac{\beta\mathbf{A}_i\mathbf{x}_i(t_m^{(n)})}{\alpha_i(t_m^{(n-1)})}.
\end{equation}
The first equality in (\ref{p6}) holds since
\begin{align}\label{}
& \bm{\lambda}_i  (t_m^{(n+1)})  - \bm{\lambda}_i(t_m^{(n)})  = \frac{\beta \mathbf{A}_i\mathbf{x}_i(t_m^{(n+1)})}{\alpha_i(t_m^{(n)})} - \frac{\beta \mathbf{A}_i\mathbf{x}_i(t_m^{(n)})}{\alpha_i(t_m^{(n-1)})} \nonumber \\
&  =  \frac{\beta \mathbf{A}_i\mathbf{x}_i(t_m^{(n+1)})}{\alpha_i(t_m^{(n)})} - \frac{\beta \mathbf{A}_i\mathbf{x}_i(t_m^{(n)})(1-\theta_i(t_m^{(n)}))}{\alpha_i(t_m^{(n)})} \nonumber \\
   & = \frac{\beta}{\alpha_i(t_m^{(n)})}(\mathbf{A}_i\mathbf{x}_i(t_m^{(n+1)}) - (1-\theta_i(t_m^{(n)}))\mathbf{A}_i\mathbf{x}_i(t_m^{(n)})). \nonumber
\end{align}
The second equality in (\ref{p6}) uses relation $\langle \mathbf{u}, \mathbf{v}\rangle = \frac{1}{2} (\parallel \mathbf{u} \parallel^2 + \parallel \mathbf{v} \parallel^2 - \parallel \mathbf{u} - \mathbf{v} \parallel^2)$ on $\Gamma_1$, $\forall \mathbf{u},\mathbf{v} \in \mathbb{R}^M$.

Then, by adding $\langle \bm{\lambda}^*, \mathbf{A}_i\mathbf{x}_i(t_m^{(n+1)})\rangle - (1 - \theta_i(t_m^{(n)})) \langle \bm{\lambda}^*, \mathbf{A}_i\mathbf{x}_i(t_m^{(n)})\rangle$ to the both sides of (\ref{p6}), we have
\begin{align}\label{p6+2}
 &{F}_i(\mathbf{x}_i(t_m^{(n+1)})) - {F}_i(\mathbf{x}_i^*) + \langle \bm{\lambda}^*, \mathbf{A}_i\mathbf{x}_i(t_m^{(n+1)})\rangle \nonumber \\
& - (1-\theta_i(t_m^{(n)}))({F}_i(\mathbf{x}_i(t_m^{(n)})) - {F}_i(\mathbf{x}_i^*) + \langle \bm{\lambda}^*, \mathbf{A}_i\mathbf{x}_i(t_m^{(n)})\rangle) \nonumber \\
 \leq & \frac{\alpha_i(t_m^{(n)})}{\beta} \langle \bm{\lambda}^*- \bm{\lambda}_i^{\mathrm{d}}(t_m), \bm{\lambda}_i(t_m^{(n+1)}) - \bm{\lambda}_i(t_m^{(n)}) \rangle   \nonumber \\
&  + \frac{1}{2}({L_i}- \frac{2-\theta_i(t_m^{(n)})}{\eta_i(t_m^{(n)})} ) \parallel \mathbf{x}_i(t_m^{(n+1)}) -\mathbf{x}_i(t_m^{(n)}) \parallel^2 \nonumber \\
&  +( \frac{\theta_i(t_m^{(n)})}{2\eta_i(t_m^{(n)})}  -\frac{\theta_i(t_m^{(n)})\mu_i}{2} )\parallel \mathbf{x}_i^*-\mathbf{x}_i(t_m^{(n)}) \parallel^2 \nonumber \\
 & - \frac{\theta_i(t_m^{(n)})}{2\eta_i(t_m^{(n)})} \parallel \mathbf{x}_i^*-\mathbf{x}_i(t_m^{(n+1)}) \parallel^2 \nonumber \\
 & + {\theta_i(t_m^{(n)})}  \langle \bm{\lambda}_i^{\mathrm{d}}(t_m),  \mathbf{A}_i \mathbf{x}_i^* \rangle  .
\end{align}
Divide the both sides of (\ref{p6+2}) by $\alpha_i(t_m^{(n)})$ and use (\ref{th12+1}) and (\ref{th14+1}), then we have
\begin{align}\label{p6+4}
 &\frac{1}{\alpha_i(t_m^{(n)})}({F}_i(\mathbf{x}_i(t_m^{(n+1)})) - {F}_i(\mathbf{x}_i^*) + \langle \bm{\lambda}^*, \mathbf{A}_i\mathbf{x}_i(t_m^{(n+1)})\rangle) \nonumber \\
& - \frac{1}{\alpha_i(t_m^{(n-1)})}({F}_i(\mathbf{x}_i(t_m^{(n)})) - {F}_i(\mathbf{x}_i^*) + \langle \bm{\lambda}^*, \mathbf{A}_i\mathbf{x}_i(t_m^{(n)})\rangle) \nonumber \\
& \leq  \frac{1}{\beta} \langle \bm{\lambda}^*- \bm{\lambda}_i^{\mathrm{d}}(t_m), \bm{\lambda}_i(t_m^{(n+1)}) - \bm{\lambda}_i(t_m^{(n)}) \rangle +  \frac{1}{2\alpha_i(t_m^{(n)})} \nonumber \\
& \quad  \cdot (L_i- \frac{2-\theta_i(t_m^{(n)})}{\eta_i(t_m^{(n)})} ) \parallel \mathbf{x}_i(t_m^{(n+1)}) -\mathbf{x}_i(t_m^{(n)}) \parallel^2 \nonumber \\
& \quad  + \frac{\theta_i(t^{(n)}_m)- \theta_i(t^{(n)}_m)\eta_i(t_m^{(n)})\mu_i}{2\alpha_i(t_m^{(n)})\eta_i(t^{(n)}_m)}  \parallel \mathbf{x}_i^*-\mathbf{x}_i(t_m^{(n)}) \parallel^2 \nonumber \\
& \quad  - \frac{\theta_i(t_m^{(n)})} {2\alpha_i(t_m^{(n)})\eta_i(t_m^{(n)})} \parallel \mathbf{x}_i^*-\mathbf{x}_i(t_m^{(n+1)}) \parallel^2 \nonumber \\
& \quad + \frac{\theta_i(t_m^{(n)})}{\alpha_i(t_m^{(n)})} \langle \bm{\lambda}_i^{\mathrm{d}}(t_m),  \mathbf{A}_i \mathbf{x}_i^* \rangle \nonumber \\
& \leq  \frac{1}{\beta} \langle \bm{\lambda}^*- \bm{\lambda}_i^{\mathrm{d}}(t_m), \bm{\lambda}_i(t_m^{(n+1)}) - \bm{\lambda}_i(t_m^{(n)}) \rangle \nonumber \\
& \quad + \frac{1}{2\alpha_i(t_m^{(n)})}(L_i- \frac{2-\theta_i(t_m^{(n)})}{\eta_i(t_m^{(n)})} ) \parallel \mathbf{x}_i(t_m^{(n+1)}) \nonumber \\
& \quad -\mathbf{x}_i(t_m^{(n)}) \parallel^2   + \frac{\theta_i(t_m^{(n-1)})}{2\alpha_i(t_m^{(n-1)})\eta_i(t_m^{(n-1)})} \parallel \mathbf{x}_i^*-\mathbf{x}_i(t_m^{(n)}) \parallel^2 \nonumber \\
& \quad  - \frac{\theta_i(t_m^{(n)})}{2\alpha_i(t_m^{(n)})\eta_i(t_m^{(n)})} \parallel \mathbf{x}_i^*-\mathbf{x}_i(t_m^{(n+1)}) \parallel^2 \nonumber \\
& \quad + \frac{\theta_i(t_m^{(n)})}{\alpha_i(t_m^{(n)})} \langle \bm{\lambda}_i^{\mathrm{d}}(t_m),  \mathbf{A}_i \mathbf{x}_i^* \rangle.
\end{align}
Then, by summing up (\ref{p6+4}) from the both sides over $n = 1,2,...,P_{i,m}$, we have
\begin{align}\label{p6+5}
& \sum_{n=1}^{P_{i,m}} ( \frac{1}{\alpha_i(t_m^{(n)})}({F}_i(\mathbf{x}_i(t_m^{(n+1)})) - {F}_i(\mathbf{x}_i^*)+ \langle \bm{\lambda}^*, \mathbf{A}_i\mathbf{x}_i(t_m^{(n+1)})\rangle) \nonumber \\
& \quad - \frac{1}{\alpha_i(t_m^{(n-1)})}({F}_i(\mathbf{x}_i(t_m^{(n)})) - {F}_i(\mathbf{x}_i^*) + \langle \bm{\lambda}^*, \mathbf{A}_i\mathbf{x}_i(t_m^{(n)})\rangle) ) \nonumber \\
& = \frac{1}{\alpha_i(t^{(P_{i,m})}_m)} ({F}_i(\mathbf{x}_i(t^{(P_{i,m}+1)}_{m})) - {F}_i(\mathbf{x}_i^*)  \nonumber \\
& \quad + \langle \bm{\lambda}^*, \mathbf{A}_i\mathbf{x}_i(t^{(P_{i,m}+1)}_{m})\rangle) \nonumber \\
& \quad - \frac{1}{\alpha_i(t^{(0)}_{m})}({F}_i(\mathbf{x}_i(t^{(1)}_{m})) - {F}_i(\mathbf{x}_i^*)  + \langle \bm{\lambda}^*, \mathbf{A}_i\mathbf{x}_i(t^{(1)}_{m})\rangle) \nonumber \\
 & = \frac{1}{\alpha_i(t_{m+1}-1)}({F}_i(\mathbf{x}_i(t_{m+1})) - {F}_i(\mathbf{x}_i^*) + \langle \bm{\lambda}^*, \mathbf{A}_i\mathbf{x}_i(t_{m+1})\rangle) \nonumber \\
& \quad - \frac{1}{\alpha_i(t_{m}-1)}({F}_i(\mathbf{x}_i(t_{m})) - {F}_i(\mathbf{x}_i^*) +  \langle \bm{\lambda}^*, \mathbf{A}_i\mathbf{x}_i(t_{m})\rangle) \nonumber \\
& \leq \sum_{n=1}^{P_{i,m}} \frac{1}{\beta} \langle \bm{\lambda}^*- \bm{\lambda}_i^{\mathrm{d}}(t_m), \bm{\lambda}_i(t_m^{(n+1)}) - \bm{\lambda}_i(t_m^{(n)}) \rangle \nonumber \\
& \quad  +  \sum_{n=1}^{P_{i,m}}  \frac{1}{2\alpha_i(t_m^{(n)})}(L_i- \frac{2-\theta_i(t_m^{(n)})}{\eta_i(t_m^{(n)})} ) \parallel \mathbf{x}_i(t_m^{(n+1)}) \nonumber \\
& \quad -\mathbf{x}_i(t_m^{(n)}) \parallel^2  + \sum_{n=1}^{P_{i,m}} \frac{\theta_i(t_m^{(n)})}{\alpha_i(t_m^{(n)})} \langle \bm{\lambda}_i^{\mathrm{d}}(t_m),  \mathbf{A}_i \mathbf{x}_i^* \rangle \nonumber \\
& \quad  +  \sum_{n=1}^{P_{i,m}} ( \frac{\theta_i(t_m^{(n-1)})}{2\alpha_i(t_m^{(n-1)})\eta_i(t_m^{(n-1)})} \parallel \mathbf{x}_i^*-\mathbf{x}_i(t_m^{(n)}) \parallel^2 \nonumber \\
&  \quad  - \frac{\theta_i(t_m^{(n)})}{2\alpha_i(t_m^{(n)})\eta_i(t_m^{(n)})} \parallel \mathbf{x}_i^*-\mathbf{x}_i(t_m^{(n+1)}) \parallel^2 ) \nonumber \\
& =  \frac{1}{\beta} \langle \bm{\lambda}^*- \bm{\lambda}_i^{\mathrm{d}}(t_m), \bm{\lambda}_i(t_{m+1}) - \bm{\lambda}_i(t_m) \rangle \nonumber \\
& \quad + \sum_{n=1}^{P_{i,m}}  \frac{1}{2\alpha_i(t^{(n)}_m)}(L_i- \frac{2-\theta_i(t^{(n)}_m)}{\eta_i(t^{(n)}_m)} )  \parallel \mathbf{x}_i(t_m^{(n+1)}) \nonumber \\
&\quad  -\mathbf{x}_i(t_m^{(n)}) \parallel^2    + \sum_{n=1}^{P_{i,m}} \frac{\theta_i(t_m^{(n)})}{\alpha_i(t_m^{(n)})} \langle \bm{\lambda}_i^{\mathrm{d}}(t_m),  \mathbf{A}_i \mathbf{x}_i^* \rangle \nonumber \\
& \quad  +  \frac{\theta_i(t_m^{(0)})}{2\alpha_i(t_m^{(0)})\eta_i(t_m^{(0)})} \parallel \mathbf{x}_i^*-\mathbf{x}_i(t_m^{(1)}) \parallel^2 \nonumber \\
&  \quad  - \frac{\theta_i(t_m^{(P_{i,m})})}{2\alpha_i(t_m^{(P_{i,m})})\eta_i(t_m^{(P_{i,m})})} \parallel \mathbf{x}_i^*-\mathbf{x}_i(t_m^{(P_{i,m}+1)}) \parallel^2  \nonumber \\
& =  \frac{1}{\beta} \langle \bm{\lambda}^*- \bm{\lambda}_i^{\mathrm{d}}(t_m), \bm{\lambda}_i(t_{m+1}) - \bm{\lambda}_i(t_m) \rangle \nonumber \\
& \quad + \sum_{n=1}^{P_{i,m}}  \frac{1}{2\alpha_i(t^{(n)}_m)}(L_i- \frac{2-\theta_i(t^{(n)}_m)}{\eta_i(t^{(n)}_m)} )  \parallel \mathbf{x}_i(t_m^{(n+1)}) \nonumber \\
&\quad  -\mathbf{x}_i(t_m^{(n)}) \parallel^2    + \sum_{n=1}^{P_{i,m}} \frac{\theta_i(t_m^{(n)})}{\alpha_i(t_m^{(n)})} \langle \bm{\lambda}_i^{\mathrm{d}}(t_m),  \mathbf{A}_i \mathbf{x}_i^* \rangle \nonumber \\
& \quad + \frac{\theta_i(t_{m}-1)}{2\alpha_i(t_{m}-1)\eta_i(t_{m}-1)} \parallel \mathbf{x}_i^*-\mathbf{x}_i(t_{m}) \parallel^2 \nonumber \\
& \quad  - \frac{\theta_i(t_{m+1}-1)}{2\alpha_i(t_{m+1}-1)\eta_i(t_{m+1}-1)} \parallel \mathbf{x}_i^*-\mathbf{x}_i(t_{m+1}) \parallel^2,
\end{align}
where we introduce intermediate variables
\begin{align}\label{}
&  \bm{\lambda}_i(t_{m}) = \bm{\lambda}_i(t_m^{(1)}) = \frac{\beta\mathbf{A}_i\mathbf{x}_i(t_m^{(1)})}{\alpha_i(t_m^{(0)})} = \frac{\beta\mathbf{A}_i\mathbf{x}_i(t_m)}{\alpha_i(t_m-1)},
\label{pa1}  \\
& \bm{\lambda}_i(t_{m+1}) =  \bm{\lambda}_i(t_m^{(P_{i,m}+1)})= \frac{\beta\mathbf{A}_i \mathbf{x}_i(t_m^{(P_{i,m}+1)})} {\alpha_i(t_m^{(P_{i,m})})} \nonumber \\
  &  \quad \quad \quad \quad = \frac{\beta\mathbf{A}_i \mathbf{x}_i(t_{m+1})} {\alpha_i(t_{m+1}-1)}. \label{pa2}
\end{align}
The equalities in (\ref{p6+5}) hold by {Propositions \ref{pp1}}, {\ref{pp2}}, and successive cancelations. (\ref{pa1}) and (\ref{pa2}) hold by {Proposition \ref{pp1}}. Then, the result of {Lemma \ref{z}} can be verified with the sixth to the last lines in (\ref{p6+5}).

\subsection{Proof of Theorem \ref{th1}}\label{tp1}

Note that (\ref{th12+1}), (\ref{1}), and (\ref{lp1}) jointly imply the synchronization of $\{\alpha_i(t_{m}-1) \}_{m \in \mathbb{N}_+}$. For convenience purpose, we define
\begin{align}\label{}
& \bm{\lambda}^{\mathrm{d}}(t_m) =  \bm{\lambda}_i^{\mathrm{d}}(t_m) = \frac{\beta \mathbf{A} \mathbf{x}^{\mathrm{d}}(t_m) }{\alpha_i(t_{m+1}-1)}
  =  \frac{\beta \mathbf{A} \mathbf{x}^{\mathrm{d}}(t_m) }{\alpha(t_{m+1}-1)}, \label{pa0} \\
& \bm{\lambda}(t_m) = \sum_{i \in \mathcal{V}}\bm{\lambda}_i(t_m) = \sum_{i \in \mathcal{V}} \frac{\beta\mathbf{A}_i\mathbf{x}_i(t_m)}{\alpha_i(t_m-1)} = \frac{\beta\mathbf{A}\mathbf{x}(t_m)}{\alpha(t_m-1)}, \label{64}
\end{align}
with the help of (\ref{cc}), $\forall i \in \mathcal{V}$. Therefore, by summing up (\ref{zt2}) over $i \in \mathcal{V}$ and $m=1,...,K$, we have

\begin{align}\label{p11}
 &\frac{1}{\alpha(t_{K+1}-1)} ({F}(\mathbf{x}(t_{K+1})) - {F}(\mathbf{x}^*) + \langle \bm{\lambda}^*, \mathbf{A}\mathbf{x}(t_{K+1})\rangle) \nonumber \\
& \quad - \frac{1}{\alpha(t_{1}-1)}({F}(\mathbf{x}(t_{1})) - {F}(\mathbf{x}^*) + \langle \bm{\lambda}^*, \mathbf{A}\mathbf{x}(t_{1})\rangle) \nonumber \\
  \leq & \frac{1}{\beta }\sum_{i \in \mathcal{V}} \sum_{m=1}^K \langle \bm{\lambda}^*- \bm{\lambda}^{\mathrm{d}}(t_m), \bm{\lambda}_i(t_{m+1}) - \bm{\lambda}_i(t_m ) \rangle \nonumber \\
 & +  \sum_{i \in \mathcal{V}} \sum_{m=1}^K  \sum_{n=1}^{P_{i,m}} \frac{1}{2\alpha_i(t^{(n)}_m)}(L_i- \frac{2-\theta_i(t^{(n)}_m)}{\eta_i(t^{(n)}_m)} ) \nonumber \\
& \cdot \parallel \mathbf{x}_i(t_m^{(n+1)}) -\mathbf{x}_i(t_m^{(n)}) \parallel^2  \nonumber \\
& + \sum_{i \in \mathcal{V}} \sum_{m=1}^K \sum_{n=1}^{P_{i,m}} \frac{\theta_i(t_m^{(n)})}{\alpha_i(t_m^{(n)})} \langle \bm{\lambda}^{\mathrm{d}}(t_m),  \mathbf{A}_i \mathbf{x}_i^* \rangle\nonumber \\
& + \sum_{i \in \mathcal{V}} \sum_{m=1}^K   ( \frac{\theta_i(t_{m}-1)}{2\alpha_i(t_{m}-1)\eta_i(t_{m}-1)} \parallel \mathbf{x}_i^*-\mathbf{x}_i(t_{m}) \parallel^2 \nonumber \\
& - \frac{\theta_i(t_{m+1}-1)}{2\alpha_i(t_{m+1}-1)\eta_i(t_{m+1}-1)} \parallel \mathbf{x}_i^*-\mathbf{x}_i(t_{m+1}) \parallel^2 ) \nonumber \\
  = & \frac{1}{\beta } \sum_{m=1}^K \underbrace{\langle \bm{\lambda}^*- \bm{\lambda}^{\mathrm{d}}(t_m), \bm{\lambda}(t_{m+1}) - \bm{\lambda}(t_m) \rangle}_{\Gamma_2} \nonumber \\
 & + \sum_{i \in \mathcal{V}} \sum_{m=1}^K  \sum_{n=1}^{P_{i,m}} \frac{1}{2\alpha_i(t^{(n)}_m)}(L_i- \frac{2-\theta_i(t^{(n)}_m)}{\eta_i(t^{(n)}_m)} ) \nonumber \\
& \cdot \parallel \mathbf{x}_i(t_m^{(n+1)}) -\mathbf{x}_i(t_m^{(n)}) \parallel^2 \nonumber \\
& +  \sum_{m=1}^K (\frac{\Xi_m}{2} \parallel \mathbf{x}^*-\mathbf{x}(t_{m}) \parallel^2    - \frac{\Xi_{m+1}}{2} \parallel \mathbf{x}^*-\mathbf{x}(t_{m+1}) \parallel^2 ) \nonumber \\
  = & \frac{1}{2\beta } \sum_{m=1}^K  (\parallel \bm{\lambda}(t_m) - \bm{\lambda}^* \parallel^2 - \parallel \bm{\lambda}(t_{m+1}) - \bm{\lambda}^* \parallel^2  \nonumber \\
& - \parallel \bm{\lambda}^{\mathrm{d}}(t_{m})- \bm{\lambda}(t_m) \parallel^2 + \underbrace{\parallel \bm{\lambda}^{\mathrm{d}}(t_{m})- \bm{\lambda}(t_{m+1}) \parallel^2}_{= \Gamma_3} ) \nonumber \\
& + \sum_{i \in \mathcal{V}} \sum_{m=1}^K  \sum_{n=1}^{P_{i,m}} \frac{1}{2\alpha_i(t^{(n)}_m)}(L_i- \frac{2-\theta_i(t^{(n)}_m)}{\eta_i(t^{(n)}_m)} )  \nonumber \\
& \cdot \parallel \mathbf{x}_i(t_m^{(n+1)}) -\mathbf{x}_i(t_m^{(n)}) \parallel^2  + \frac{\Xi_1}{2} \parallel \mathbf{x}^*-\mathbf{x}(t_{1}) \parallel^2  \nonumber \\
&  - \frac{\Xi_{K+1}}{2} \parallel \mathbf{x}^*-\mathbf{x}(t_{K+1}) \parallel^2 \nonumber \\
\leq & \frac{1}{2\beta } \sum_{m=1}^K   (\parallel \bm{\lambda}(t_m) - \bm{\lambda}^* \parallel^2 - \parallel \bm{\lambda}(t_{m+1}) - \bm{\lambda}^* \parallel^2 ) \nonumber \\
& + \sum_{i \in \mathcal{V}} \sum_{m=1}^K  \sum_{n=1}^{P_{i,m}} \frac{1}{2\alpha_i(t^{(n)}_m)}(L_i- \frac{2-\theta_i(t^{(n)}_m)}{\eta_i(t^{(n)}_m)} )  \nonumber \\
& \cdot \parallel \mathbf{x}_i(t_m^{(n+1)}) -\mathbf{x}_i(t_m^{(n)}) \parallel^2  + \frac{\Xi_1}{2} \parallel \mathbf{x}^*-\mathbf{x}(t_{1}) \parallel^2  \nonumber \\
&  - \frac{\Xi_{K+1}}{2} \parallel \mathbf{x}^*-\mathbf{x}(t_{K+1}) \parallel^2   \nonumber \\
 + & \sum_{i \in \mathcal{V}} \sum_{m=1}^K  \sum_{n=1}^{P_{i,m}} \frac{ (H + D)\beta \parallel \mathbf{A} \parallel^2}{\alpha^2(t_{m+2}-1)} \parallel \mathbf{x}_i(t_m^{(n+1)}) - \mathbf{x}_i(t_m^{(n)}) \parallel^2 \nonumber \\
 & +   \sum_{i \in \mathcal{V}} \sum_{n=1}^{P_{i,0}} \frac{D\beta \parallel \mathbf{A} \parallel^2}{\alpha^2(t_2-1)} \parallel \mathbf{x}_i(t_{0}^{(n+1)}) - \mathbf{x}_i(t_{0}^{(n)}) \parallel^2 \nonumber \\
  = & \frac{1}{2\beta } (\parallel \bm{\lambda}(t_1) - \bm{\lambda}^* \parallel^2 - \parallel \bm{\lambda}(t_{K+1}) - \bm{\lambda}^* \parallel^2 ) \nonumber \\
& + \frac{\Xi_1}{2} \parallel \mathbf{x}^*-\mathbf{x}(t_{1}) \parallel^2    - \frac{\Xi_{K+1}}{2} \parallel \mathbf{x}^*-\mathbf{x}(t_{K+1}) \parallel^2   \nonumber \\
&  + \sum_{i \in \mathcal{V}} \sum_{m=1}^K  \sum_{n=1}^{P_{i,m}} ( \frac{L_i}{2\alpha_i(t^{(n)}_m)} - \frac{2-\theta_i(t^{(n)}_m)}{2\alpha_i(t^{(n)}_m)\eta_i(t^{(n)}_m)} \nonumber \\
  & +  \frac{(H + D) \beta\parallel \mathbf{A} \parallel^2}{\alpha^2(t_{m+2}-1)} ) \parallel \mathbf{x}_i(t_m^{(n+1)}) -\mathbf{x}_i(t_m^{(n)}) \parallel^2 \nonumber \\
 & +   \sum_{i \in \mathcal{V}} \sum_{n=1}^{P_{i,0}} \frac{D\beta \parallel \mathbf{A} \parallel^2}{\alpha^2(t_2-1)} \parallel \mathbf{x}_i(t_{0}^{(n+1)}) - \mathbf{x}_i(t_{0}^{(n)}) \parallel^2 \nonumber \\
  \leq  & \frac{1}{2\beta } (\parallel \bm{\lambda}(t_1) - \bm{\lambda}^* \parallel^2 - \parallel \bm{\lambda}(t_{K+1}) - \bm{\lambda}^* \parallel^2 ) \nonumber \\
 & +  \frac{\Xi_1}{2} \parallel \mathbf{x}^*-\mathbf{x}(t_{1}) \parallel^2    - \frac{\Xi_{K+1}}{2} \parallel \mathbf{x}^*-\mathbf{x}(t_{K+1}) \parallel^2   \nonumber \\
 & +  \sum_{i \in \mathcal{V}} \sum_{n=1}^{P_{i,0}} \frac{D \beta \parallel \mathbf{A} \parallel^2}{\alpha^2(t_{2}-1)} \parallel \mathbf{x}_i(t_{0}^{(n+1)}) - \mathbf{x}_i(t_{0}^{(n)}) \parallel^2.
\end{align}
In the first equality, (\ref{271}) is applied and the third term is cancelled out due to
\begin{align}\label{}
& \sum_{i \in \mathcal{V}} \sum_{m=1}^K \sum_{n=1}^{P_{i,m}} \frac{\theta_i(t_m^{(n)})}{\alpha_i(t_m^{(n)})} \langle \bm{\lambda}^{\mathrm{d}}(t_m),  \mathbf{A}_i \mathbf{x}_i^* \rangle \nonumber \\
& = \sum_{i \in \mathcal{V}} \sum_{m=1}^K \langle \bm{\lambda}^{\mathrm{d}}(t_m),  \mathbf{A}_i \mathbf{x}_i^* \rangle  \nonumber \\
& =  \sum_{m=1}^K \langle \bm{\lambda}^{\mathrm{d}}(t_m),  \mathbf{A} \mathbf{x}^* \rangle =0. \nonumber
\end{align}
The second equality in (\ref{p11}) holds by performing successive cancellations and using the relation $\langle \mathbf{u}-\mathbf{v}, \mathbf{r}-\mathbf{s} \rangle = \frac{1}{2} (\parallel \mathbf{u}-\mathbf{s}\parallel^2 - \parallel \mathbf{u}-\mathbf{r}\parallel^2 + \parallel \mathbf{v}-\mathbf{r}\parallel^2 - \parallel \mathbf{v}-\mathbf{s} \parallel^2)$ on $\Gamma_2$, $\forall \mathbf{u},\mathbf{v},\mathbf{r},\mathbf{s} \in \mathbb{R}^B$. The second inequality in (\ref{p11}) holds with
\begin{align}\label{p11+1}
& \sum_{m=1}^K  \Gamma_3   = \sum_{m=1}^K \| \frac{\beta\mathbf{A}\mathbf{x}(t_{m+1})}{\alpha(t_{m+1}-1)}- \frac{\beta \mathbf{A}\mathbf{x}^{\mathrm{d}}(t_{m})}{\alpha(t_{m+1}-1)} \|^2 \nonumber \\
& \leq \sum_{m=1}^K \frac{\beta^2 \parallel \mathbf{A} \parallel^2}{\alpha^2(t_{m+1}-1)} \parallel (\mathbf{x}(t_{m+1}) - \mathbf{x}(t_{m})) \nonumber \\
& \quad  + (\mathbf{x}(t_{m}) - \mathbf{x}^{\mathrm{d}}(t_m)) \parallel^2  \nonumber \\
& \leq \sum_{m=1}^K \frac{2 \beta^2 \parallel \mathbf{A} \parallel^2}{\alpha^2(t_{m+1}-1)} (\parallel \mathbf{x}(t_{m+1}) - \mathbf{x}(t_{m}) \parallel^2 \nonumber \\
& \quad +  \parallel \mathbf{x}(t_{m}) - \mathbf{x}^{\mathrm{d}}(t_m) \parallel^2)  \nonumber \\
& \leq  \sum_{i \in \mathcal{V}} \sum_{m=1}^K \frac{2 \beta^2 \parallel \mathbf{A} \parallel^2}{\alpha^2(t_{m+1}-1)} ( H \sum_{n=1}^{P_{i,m}}  \parallel \mathbf{x}_i(t^{(n+1)}_{m}) - \mathbf{x}_i(t^{(n)}_{m}) \parallel^2 \nonumber \\
 & \quad + D\sum_{n = 1}^{P_{i,m-1}} \parallel \mathbf{x}_i (t^{(n+1)}_{m-1}) -  \mathbf{x}_i(t^{(n)}_{m-1}) \parallel^2 ) \nonumber \\
 & \leq \sum_{i \in \mathcal{V}} \sum_{m=1}^K \frac{2 \beta^2 \parallel \mathbf{A} \parallel^2 H}{\alpha^2(t_{m+2}-1)}  \sum_{n=1}^{P_{i,m}}  \parallel \mathbf{x}_i(t^{(n+1)}_{m}) - \mathbf{x}_i(t^{(n)}_{m}) \parallel^2 \nonumber \\
 & + \sum_{i \in \mathcal{V}} \sum_{m=0}^{K-1} \frac{2 \beta^2 \parallel \mathbf{A} \parallel^2 D }{\alpha^2(t_{m+2}-1)} \sum_{n = 1}^{P_{i,m}} \parallel \mathbf{x}_i (t^{(n+1)}_{m})  -  \mathbf{x}_i(t^{(n)}_{m}) \parallel^2 \nonumber \\
 & \leq  \sum_{i \in \mathcal{V}} \sum_{m=1}^K \sum_{n = 1}^{P_{i,m}} \frac{2\beta^2 \parallel \mathbf{A} \parallel^2 (H + D) }{\alpha^2(t_{m+2}-1)} \nonumber \\
&  \quad \cdot \parallel \mathbf{x}_i (t^{(n+1)}_m) -  \mathbf{x}_i(t^{(n)}_m) \parallel^2 \nonumber \\
& +  \sum_{i \in \mathcal{V}} \sum_{n=1}^{P_{i,0}} \frac{2 \beta^2 \parallel \mathbf{A} \parallel^2D }{\alpha^2(t_{2}-1)} \parallel \mathbf{x}_i(t_{0}^{(n+1)}) - \mathbf{x}_i(t_{0}^{(n)}) \parallel^2,
\end{align}
where the third inequality holds with {Proposition \ref{la2}} and the forth inequality holds with $\alpha(t_{m+2}-1) <\alpha(t_{m+1}-1)$ (see {Proposition \ref{pr1}}).

The last inequality in (\ref{p11}) holds with
\begin{align}\label{}
& \frac{L_i}{2} - \frac{2-\theta_i(t^{(n)}_m)}{2\eta_i(t^{(n)}_m)} + \frac{\alpha_i(t^{(n)}_m)(H + D) \beta\parallel \mathbf{A} \parallel^2}{\alpha^2(t_{m+2}-1)} \nonumber \\
& < \frac{L_i}{2} - \frac{1}{2\eta_i(t^{(n)}_m)} + \frac{(H + D) \beta \Pi \parallel \mathbf{A} \parallel^2}{\alpha(t_{m+2}-1)} \leq 0,
\end{align}
where (\ref{17-5}), (\ref{th13}) and $\theta_i(t^{(n)}_m) \in (0,1)$ are considered.

Then, with $\Delta_1$ defined in (\ref{c1}), (\ref{p11}) can be rearranged as
\begin{align}\label{43}
&\frac{1}{\alpha(t_{K+1}-1)} ({F}(\mathbf{x}(t_{K+1})) - {F}(\mathbf{x}^*) + \langle \bm{\lambda}^*, \mathbf{A}\mathbf{x}(t_{K+1})\rangle) \nonumber \\
& + \frac{1}{2\beta} \parallel \bm{\lambda}(t_{K+1}) - \bm{\lambda}^* \parallel^2 + \frac{\Xi_{K+1}}{2} \parallel \mathbf{x}^*-\mathbf{x}(t_{K+1}) \parallel^2  \leq \Delta_1.
\end{align}
Hence, with the help of (\ref{sd1}), we have $0 \leq {F}(\mathbf{x}(t_{K+1})) - {F}(\mathbf{x}^*) + \langle \bm{\lambda}^*, \mathbf{A}\mathbf{x}(t_{K+1})\rangle \leq \alpha(t_{K+1}-1) \Delta_1$ and $\frac{1}{2\beta} \parallel \bm{\lambda}(t_{K+1}) - \bm{\lambda}^* \parallel^2 \leq \Delta_1$. Therefore, by the definition of $\bm{\lambda}(t_{K+1})$ in (\ref{64}), we have
\begin{align}\label{}
\frac{\beta}{ \alpha(t_{K+1}-1)} \parallel \mathbf{A}\mathbf{x}(t_{K+1})\parallel  \leq & \parallel \bm{\lambda}(t_{K+1}) - \bm{\lambda}^* \parallel + \parallel \bm{\lambda}^* \parallel \nonumber \\
 \leq & \sqrt{2\beta \Delta_1} +\parallel \bm{\lambda}^* \parallel,
\end{align}
which gives
\begin{align}\label{fh1}
\parallel \mathbf{A}\mathbf{x}(t_{K+1}) \parallel \leq \frac{\sqrt{2\beta \Delta_1} + \parallel \bm{\lambda}^* \parallel }{\beta}\alpha(t_{K+1}-1).
\end{align}

On the other hand,
\begin{align}\label{fh2}
& F(\mathbf{x} (t_{K+1}))  - F(\mathbf{x}^*) \leq  \Delta_1 \alpha(t_{K+1}-1) - \langle \bm{\lambda}^*, \mathbf{A} \mathbf{x}(t_{K+1})\rangle \nonumber \\
& \leq  \Delta_1 \alpha(t_{K+1}-1) + \parallel \bm{\lambda}^*\parallel \parallel \mathbf{A} \mathbf{x}(t_{K+1}) \parallel \nonumber \\
& \leq  ( \Delta_1 + \frac{\sqrt{2\beta \Delta_1} + \parallel \bm{\lambda}^* \parallel }{\beta} \parallel \bm{\lambda}^* \parallel ) \alpha(t_{K+1}-1),
\end{align}
and
\begin{align}\label{fh3}
& F(\mathbf{x} (t_{K+1}))  - F(\mathbf{x}^*) \geq  - \parallel \bm{\lambda}^* \parallel \parallel \mathbf{A} \mathbf{x}(t_{K+1}) \parallel \nonumber \\
& \geq  -  \frac{\sqrt{2\beta \Delta_1} + \parallel \bm{\lambda}^* \parallel }{\beta}\parallel \bm{\lambda}^* \parallel\alpha (t_{K+1}-1).
\end{align}
By combining (\ref{fh1}), (\ref{fh2}) and (\ref{fh3}), the proof is completed.

\subsection{Proof of Lemma \ref{l5}}\label{l5p}

Note that by (\ref{th12+1}), (\ref{1}) and (\ref{lp1}), $\{\alpha_i(t_{m}-1) \}_{m \in \mathbb{N}_+}$ is synchronized, i.e, (\ref{tha1}) holds. Then, (\ref{llp}) can be proved with
\begin{align}\label{llp2}
& \frac{\eta_i(t_{m}-1)}{\eta_j(t_{m}-1)} = \frac{\eta_i(t^{(P_{i,m-1})}_{m-1})}{\eta_j(t^{(P_{j,m-1})}_{m-1})} \nonumber \\
& =  \frac{P _{j,m-1} ( Q_{m-1} + \frac{2(H+D)\beta \Pi \parallel \mathbf{A} \parallel^2}{\alpha_j(t_{m+1}-1)} )}{P _{i,m-1} ( Q_{m-1} + \frac{2(H+D)\beta \Pi \parallel \mathbf{A} \parallel^2}{\alpha_i(t_{m+1}-1)} )}  \nonumber \\
& =  \frac{P_{j,m-1}}{P_{i,m-1}},
\end{align}
where {Proposition \ref{pp1}}, (\ref{cc}), and (\ref{16}) are used.

(\ref{th13}) can be directly verified by (\ref{16}) with $P_{i,m} \geq 1$, $Q_m \geq L^{\mathrm{g}} \geq L_i$, and (\ref{cc}).

To prove (\ref{151}), by (\ref{cc}), (\ref{lp1}) and (\ref{16}), we have
\begin{align}\label{73}
& \frac{\theta_i(t^{(n)}_m)}{\alpha_i(t^{(n)}_m)\eta_i(t^{(n)}_m)} = \frac{1}{P_{i,m}\eta_i(t^{(n)}_m)}  \nonumber \\
 &  =  Q_{m} + \frac{2(H+D)\beta \Pi \parallel \mathbf{A} \parallel^2}{\alpha(t_{m+2}-1)}.
\end{align}
Hence, if $n=2,3,...,P_{i,m}$, (\ref{151}) holds with the left-hand side being 0. If $n=1$, then
\begin{align}\label{45}
 & \frac{\theta_i(t^{(1)}_m)}{\eta_i(t^{(1)}_m)\alpha_i(t^{(1)}_m)} -  \frac{\theta_i(t^{(0)}_m)}{\eta_i(t^{(0)}_m)\alpha_i(t^{(0)}_m)} \nonumber \\
 & =  \frac{\theta_i(t^{(1)}_m)}{\eta_i(t^{(1)}_m) \alpha_i(t^{(1)}_m)} -  \frac{\theta_i(t^{(P_{i,m-1})}_{m-1})}{\eta_i(t^{(P_{i,m-1})}_{m-1}) \alpha_i(t^{(P_{i,m-1})}_{m-1})} \nonumber \\
 & =  Q_{m} - Q_{m-1} + 2(H+D)\beta \Pi \parallel \mathbf{A} \parallel^2 \nonumber \\
  & \quad \cdot ( \frac{1}{\alpha(t_{m+2}-1)}  -  \frac{1}{\alpha(t_{m+1}-1)} )  \nonumber \\
& =  Q_{m} - Q_{m-1} + 2(H+D)\beta \Pi \parallel \mathbf{A} \parallel^2,
\end{align}
where the last two equalities use formulas (\ref{73}) and (\ref{17-1}), respectively. Therefore, (\ref{151}) holds.

\subsection{Proof of Theorem \ref{la1}}\label{tp2}
With a given $\epsilon>0$, we let $\alpha(t_{K+1}-1) \leq \epsilon$, which means $\frac{\alpha(t_{1}-1)}{K\alpha(t_{1}-1)+1}  \leq \epsilon$ (by (\ref{17-1})).
By solving $K$, we have (\ref{27}).

To prove the convergence, one still needs to satisfy condition (\ref{th14+1}). In slot $m$, by (\ref{lp1}), (\ref{th14+1}) can be rewritten as
\begin{align}\label{17}
  & \frac{\theta_i(t^{(n)}_m)}{\eta_i(t^{(n)}_m) \alpha_i(t^{(n)}_m)} - \frac{\theta_i(t^{(n-1)}_{m})}{\eta_i(t^{(n-1)}_{m}) \alpha_i(t^{(n-1)}_{m})} \nonumber \\
   &  \leq \frac{ \theta_i(t^{(n)}_m)\mu_i}{\alpha_i(t^{(n)}_m)} = \frac{ \mu_i}{P_{i,m}}.
\end{align}
Considering $P_{i,m} \leq H$ and $\mu \leq \mu_i$, a more conservative condition of (\ref{17}) can be obtained as
\begin{equation}\label{17+1}
  \frac{\theta_i(t^{(n)}_m)}{\eta_i(t^{(n)}_m) \alpha_i(t^{(n)}_m)} - \frac{\theta_i(t^{(n-1)}_{m})}{\eta_i(t^{(n-1)}_{m}) \alpha_i(t^{(n-1)}_{m})} \leq \frac{\mu}{H}.
\end{equation}
To ensure (\ref{17+1}) to hold, we combine (\ref{151}) and (\ref{17+1}), which gives (considering $\frac{\mu}{H} >0$ unconditionally)
\begin{align}\label{cd}
Q_m - & Q_{m-1} + 2(H+D)\beta \Pi \parallel \mathbf{A} \parallel^2 \leq \frac{ \mu}{H}.
\end{align}
Solving the requirement for $Q_m$ and $\beta$ in (\ref{cd}) gives (\ref{35}) and (\ref{32-1}).

By now, all the conditions in Theorem \ref{th1} are satisfied by those in Theorem 2. By recalling $\alpha(t_{K+1}-1) \leq \epsilon$, (\ref{f2}) and (\ref{f3}) can be written into (\ref{22+1}) and (\ref{22+2}), respectively. In addition, as seen from (\ref{17-1}), $\alpha(t_{K+1}-1)$ is with an order of $\mathcal{O}(\frac{1}{K})$. Hence, the results (\ref{f2}) and (\ref{f3}) can be further written into (\ref{23}) and (\ref{24}), respectively. This completes the proof.

\small

\bibliographystyle{IEEEtran}

\bibliography{1myref}

\end{document}